\documentclass{amsart}
\usepackage{graphicx}

\usepackage[english]{babel}

\usepackage{amsmath, amsfonts, amssymb, amsthm, graphicx, color, url}
\usepackage[inline,shortlabels]{enumitem}
\usepackage[justification=centering, font={sl,small}, labelfont={bf}, margin=1cm]{caption}
\usepackage{blkarray}
\usepackage{overpic} 
\usepackage{tikz}
\usepackage{hyperref}
\newcommand*\circled[1]{\tikz[baseline=(char.base)]{
            \node[shape=circle,draw,inner sep=1pt] (char) {#1};}}

\hypersetup{colorlinks=true, citecolor=darkblue, linkcolor=darkblue}

\makeatletter
\def\namedlabel#1#2{\begingroup
   \def\@currentlabel{#2}%
   \label{#1}\endgroup
}
\makeatother

\makeatletter
\renewcommand{\pod}[1]{\mathchoice
  {\allowbreak \if@display \mkern 6mu\else \mkern 6mu\fi (#1)}
  {\allowbreak \if@display \mkern 6mu\else \mkern 6mu\fi (#1)}
  {\mkern4mu(#1)}
  {\mkern4mu(#1)}
}

\makeatletter
\renewcommand\@makefntext[1]{%
  \noindent\makebox[0em][r]{\@makefnmark}#1}
\makeatother

\newtheorem{theorem}{Theorem}[section]
\newtheorem{proposition}[theorem]{Proposition}
\newtheorem{corollary}[theorem]{Corollary}
\newtheorem{lemma}[theorem]{Lemma}

\theoremstyle{definition}
\newtheorem{definition}[theorem]{Definition}
\newtheorem{example}[theorem]{Example}
\newtheorem{conjecture}[theorem]{Conjecture}

\newtheorem{remark}[theorem]{Remark}
\newtheorem{question}[theorem]{Question}

\makeatletter
\newtheorem*{rep@theorem}{\rep@title}
\newcommand{\newreptheorem}[2]{%
\newenvironment{rep#1}[1]{%
 \def\rep@title{#2 \ref{##1}}%
 \begin{rep@theorem}}%
 {\end{rep@theorem}}}
\makeatother

\newreptheorem{theorem}{Theorem}

\newcommand{\cL}{\mathcal{L}}

\newcommand{\cR}{\mathcal{R}}

\newcommand{\frS}{\mathfrak{S}}

\newcommand{\sfb}{\mathsf{b}}

\newcommand{\sfN}{\mathsf{N}}
\newcommand{\sfE}{\mathsf{E}}

\newcommand{\F}{F} 
\newcommand{\A}{A} 
\newcommand{\X}{X} 
\newcommand{\Y}{Y} 

\newcommand{\subwordComplex}[2]{\mathcal{SC}(#1,#2)} 
\newcommand{\ssm}{\smallsetminus} 

\newcommand{\dobigast}[1]{%
  \vcenter{#1\kern.2ex\hbox{$\ast$}\kern.2ex}}

\newcommand{\productp}[2]{\Delta_{{#1}}\times\Delta_{{#2}}} 

\newcommand{\Tam}[1]{\operatorname{Tam}_{#1}}

\def\horiz{\operatorname{horiz}}
\def\min{{\operatorname{min}}}
\def\max{{\operatorname{max}}}

\def\hroot{\operatorname{hroot}} 

\newcommand{\knu}{(k,\nu)}

\newcommand{\reverse}[1]{\overleftarrow{#1}}

\newcommand{\ol}[1]{\overline{#1}}

\definecolor{darkblue}{rgb}{0,0,0.7} 
\newcommand{\darkblue}{\color{darkblue}} 
\newcommand{\defn}[1]{\emph{\darkblue #1}} 

\usepackage{todonotes}

\newcommand{\cross}[1][black]{\raisebox{-.15cm}{\includegraphics[scale=.9]{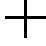}}}
\newcommand{\elbow}[1][black]{\raisebox{-.15cm}{\includegraphics[scale=.9]{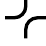}}}
\newcommand{\Tmin}{T_\min} 
\newcommand{\Tmax}{T_\max} 
\newcommand{\ncomp}{\nsim} 
\newcommand{\bmin}{\sfb^\min} 
\newcommand{\rflush}{\cR} 
\newcommand{\lflush}{\cL} 
\newcommand{\EG}{\mathcal{EG}} 

\newsavebox{\smlmat}
\savebox{\smlmat}{$\left(\begin{smallmatrix}
1 & 1 & 1 & 1 & 2 & 2 & 3 & 3 & 3 & 4 & 4 & 5 \\
5 & 4 & 3 & 2 & 6 & 4 & 5 & 4 & 3 & 6 & 4 & 5
\end{smallmatrix}\right)$}

\graphicspath{{Figures/}}

\title[$\nu$-Tamari lattice, $\nu$-trees, $\nu$-bracket vectors, subword complexes]{The $\nu$-Tamari lattice via $\nu$-trees, $\nu$-bracket vectors, and subword complexes}

\author{Cesar Ceballos}
\address{Faculty of Mathematics, University of Vienna, Vienna, Austria}

\author{Arnau Padrol}
\address{Sorbonne Universit\'e, Institut de Math\'ematiques de Jussieu - Paris Rive Gauche (UMR 7586), Paris, France
} 

\author{Camilo Sarmiento}
\address{Departamento de Matem\'aticas y Estad\'istica, Universidad del Norte, Barranquilla, Colombia}

\thanks{%
This project was partially supported by the project ``Austria/France Scientific \& Technological Cooperation", which is co-financed by the Austrian Federal Ministry of Science, Research and Economy BMWFW (Project No. FR 10/2018) and by the French Ministry of Foreign Affairs and International Development (PHC Amadeus 2018 Project No. 39444WJ).
The research of C.C. was also supported by the Austrian Science Foundation FWF, grant F 5008-N15, in the framework of the Special Research Program ``Algorithmic and Enumerative Combinatorics''; A.P. was also supported by the grant ANR-17-CE40-0018 of the French National Research Agency ANR (project CAPPS), as well as the program PEPS Jeunes Chercheur-e-s from the INSMI
}

\begin{document}

\begin{abstract}
We give new interpretations of the $\nu$-Tamari lattice of Pr\'eville-Ratelle and Viennot. First, we describe it as a rotation lattice of $\nu$-trees, which uncovers the relation with known combinatorial objects such as \mbox{tree-like} tableaux and  north-east fillings.
Then, using a formulation in terms of bracket vectors of $\nu$-trees and componentwise order, we provide a simple description of the lattice property. 
We also show that the {$\nu$-Tamari} lattice is isomorphic to the increasing-flip poset of a suitably chosen subword complex, and settle a special case of Rubey's lattice conjecture concerning the poset of pipe dreams defined by chute moves.
Finally, this point of view generalizes to multi $\nu$-Tamari complexes, and gives (conjectural) insight on their geometric realizability via polytopal subdivisions of multiassociahedra.
\end{abstract}

\maketitle

\section{Introduction}

The $\nu$-Tamari lattice is a partial order on the set of lattice paths weakly above a given path $\nu$ that generalizes the Dyck/ballot-path formulation of the classical Tamari lattice~\cite{TamariFestschrift,TamariPhD}. It has been recently introduced by Pr\'eville-Ratelle and Viennot~\cite{PrevilleRatelleViennot2014} as a further generalization of the $m$-Tamari lattice on Fuss-Catalan paths, which was first considered by F.~Bergeron and Pr\'eville-Ratelle in connection to the combinatorics of higher diagonal coinvariant spaces~\cite{BergeronPrevilleRatelle2012}. 
These lattices have attracted considerable attention in other areas such as representation theory and Hopf algebras~\cite{BCP2013,ceballos017,NovelliThibon2014}, and remarkable enumerative, algebraic, combinatorial and geometric properties have been discovered~\cite{Bergeron2012,BFP2011,tams2017,chapoton_sur_2005,fang_enumeration_2017}.

\begin{figure}[htbp]
\begin{center}
 \makebox[\textwidth][c]{\includegraphics[width=1.2\textwidth]{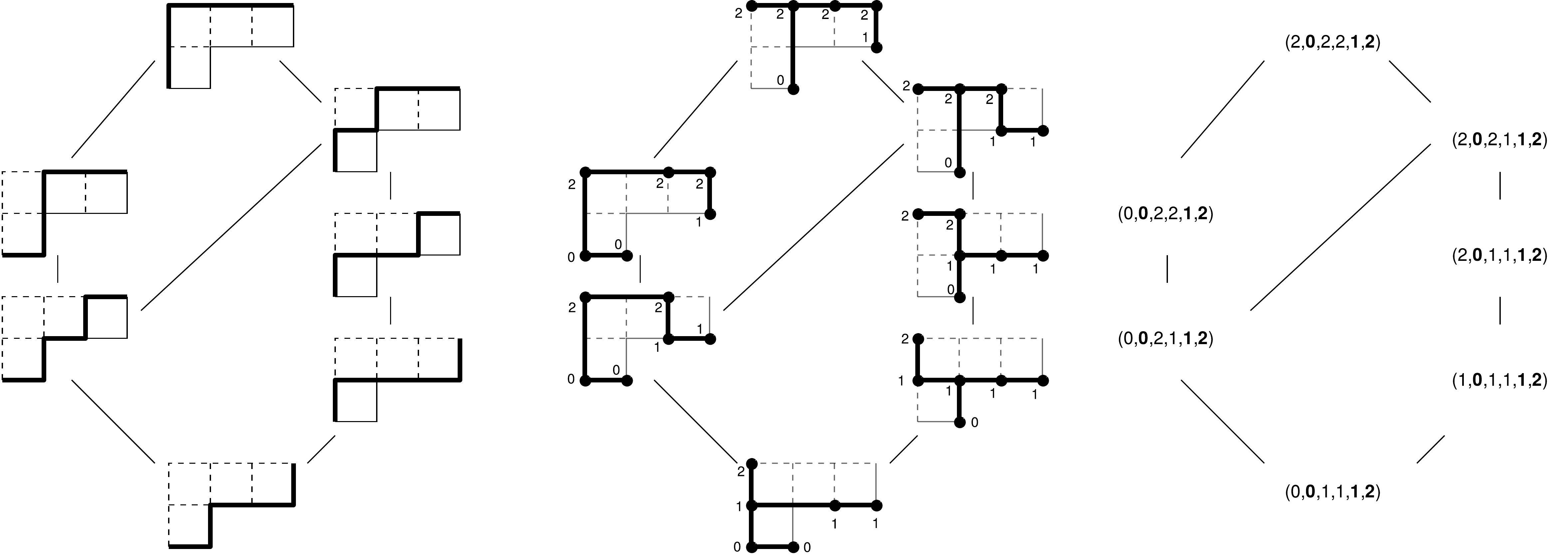}}
\end{center}
\caption{The $\nu$-Tamari lattice for $\nu=\sfE\sfN\sfE\sfE\sfN$ (left). The rotation lattice of $\nu$-trees (middle). The lattice of $\nu$-bracket vectors (right).}
\label{fig:Tamari35}
\end{figure}

In this paper, we present a new formulation of the $\nu$-Tamari lattice as a rotation lattice of $\nu$-trees (Theorem~\ref{thm:rotationlatticetrees}), which specializes to the rotation lattice of binary trees in the Catalan case $\nu=(\sfN\sfE)^n$. 
Our $\nu$-trees are the $k=1$ case of Jonsson's \defn{$(k+1)$-diagonal-free} maximal subsets of a Ferrers diagram associated with~$\nu$~\cite{Jonsson2005}. 
These were later studied by Serrano and Stump under the name of \defn{$k$-north-east fillings} of Ferrers diagrams~\cite{serrano_maximal_2012}, who showed that they can be realized as pipe dreams and facets of certain subword complexes. This connection for the classical Catalan case had been shown earlier by Woo in~\cite{woo_catalan_2004}.
Additionally, $\nu$-trees are also equivalent to \defn{non-crossing tree-like tableaux} (Proposition~\ref{prop:nutreesaretreelike}), introduced by Aval, Bossicault, and Nadeau in~\cite{AvalBoussicaultNadeau2013} 
in connection to permutations and alternative tableaux~\cite{CorteelNadeau2009,CorteelWilliams2011,Postnikov2006,SteingrimssonWilliams2007}, and their Catalan subfamilies~\cite{AvalViennot2010,viennot2007}. Furthermore, the $\nu$-Tamari lattice is one of Viennot's recently presented \defn{Maule posets}~\cite{ViennotMaule2017}.

The presentation in terms of $\nu$-trees provides new interpretations of properties of the $\nu$-Tamari lattice. 
In particular, we show that the $\nu$-Tamari lattice is isomorphic to a lattice of \defn{$\nu$-bracket vectors} under componentwise order (Theorem~\ref{thm:bracketlattice}). 
This leads to a description of the meet of two $\nu$-bracket vectors as their componentwise minimum (Proposition~\ref{prop:meet}), providing a simple proof of the lattice property similar to the one shown by Huang and Tamari in their original ``simple proof of the lattice property''~\cite{HuangTamari1972}. 
An example of Theorem~\ref{thm:rotationlatticetrees} and Theorem~\ref{thm:bracketlattice} is illustrated in Figure~\ref{fig:Tamari35}. The equivalence of the $\nu$-Tamari lattice and the rotation lattice of $\nu$-trees follows from a \defn{flushing} bijection between $\nu$-trees and $\nu$-paths (Proposition~\ref{prop_flushingbijection}).
This bijection is shown to be equivalent to (a slight generalization of) a bijection between certain pipe dreams and Dyck paths presented by Woo in~\cite[Section~3]{woo_catalan_2004}, which is described in terms of the Edelman-Greene correspondence (Proposition~\ref{prop:EdelmanGreene}).

Moreover, our results imply that the $\nu$-Tamari lattice is isomorphic to the increasing flip poset of
a suitably chosen subword complex (Theorem~\ref{thm:TamariSubwordComplex}). Subword complexes are simplicial complexes introduced by Knutson and Miller in their study of the Gr\"obner geometry of Schubert varieties~\cite{knutson_subword_2004,knutson_grobner_2005}.
Our result generalizes a known result of Pilaud and Pocchiola for the classical Tamari lattice~\cite[Section~3.3 and Theorem~23]{PilaudPocchiola}, which has been rediscovered by Stump~\cite{Stump} and Stump and Serrano in~\cite{serrano_maximal_2012}, and which follows from Woo's bijection in~\cite{woo_catalan_2004}.
As a consequence of our result, we settle a special case of Rubey's Lattice Conjecture~\cite[Conjecture~2.8]{rubey_maximal_2012}, which affirms that a poset of reduced pipe dreams defined by (general) chute moves has the structure of a lattice (Theorem~\ref{thm_RubeysConjecture}).

The relation between the Tamari lattice and subword complexes has inspired further connections with pseudotriangulation polytopes~\cite{RoteSantosStreinu-polytopePseudotriangulations,RoteSantosStreinu-survey}, cluster algebras~\cite{CeballosPilaud,FominZelevinsky-ClusterAlgebrasI,FominZelevinsky-ClusterAlgebrasII}, Hopf algebras~\cite{ceballos_Hopf_2017}, and multiassociahedra~\cite{Jonsson2005,soll_type-b_2009}. 
Concerning the latter, the definition of the multiassociahedron can be naturally generalized to $\nu$-trees, giving rise to the $(k, \nu)$-Tamari complex, which is also a subword complex (these are the complexes of $(k+1)$-diagonal free subsets and $k$-north-east fillings considered in~\cite{Jonsson2005} and~\cite{serrano_maximal_2012}). 
For special choices of $k$ and $\nu$, we show that  the facet adjacency graph of the $(k,\nu)$-Tamari complex can be realized as the edge graph of a polytopal subdivision of a multiassociahedron (Proposition~\ref{prop:multi}), partially extending previous results for the case $k=1$~\cite{tams2017}. It would be interesting to know whether a similar result might hold for more general $k$ and~$\nu$ (Question~\ref{que:multisubd}).

\section{The rotation lattice of \texorpdfstring{$\nu$}{v}-trees}\label{sec:rotationlatticevtrees}

Let $\nu$ be a lattice path on the plane consisting of a finite number of north and east unit steps. We denote by $\F_\nu$ the Ferrers diagram that lies weakly above $\nu$ inside the smallest rectangle containing~$\nu$, and by $\A_\nu$ the set of lattice points weakly above~$\nu$ in~$\F_\nu$. 
We say that $p,q\in \A_\nu$ are \defn{$\nu$-incompatible}, and write $p \ncomp q$,  if and only if~$p$ is southwest or northeast to $q$ and the smallest rectangle containing $p$ and $q$ lies entirely inside $\F_\nu$. Otherwise, we say that $p$ and $q$ are $\nu$-compatible;  see Figure~\ref{fig:vcompatibility}.

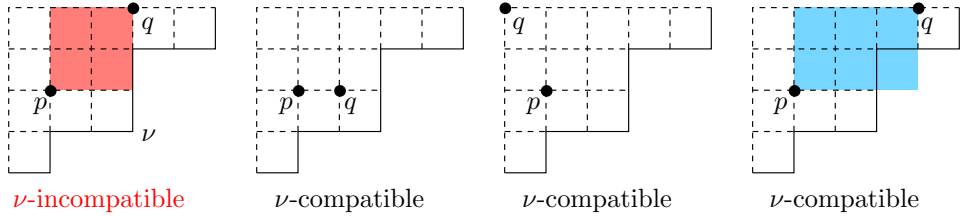
\begin{figure}[htbp]
	\begin{overpic}[width=1\textwidth]{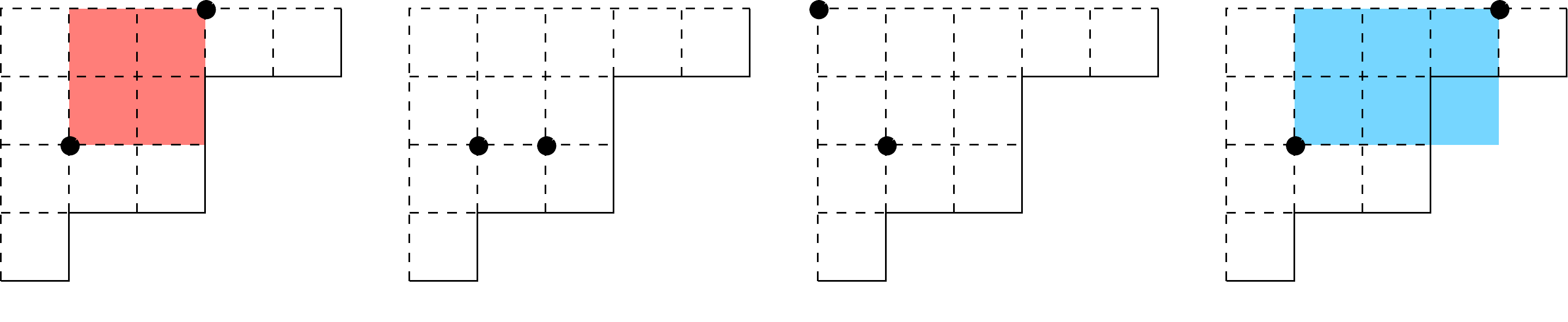}
		\put(14,6){$\nu$}
		\put(2.7,9.3){$p$}
		\put(14,17.7){$q$}
		\put(.5,-1){\color{red} $\nu$-incompatible}
		\put(28.5,9.3){$p$}
		\put(35.3,9.3){$q$}
		\put(28,-1){$\nu$-compatible}
		\put(54.5,9.3){$p$}
		\put(53,17.7){$q$}
		\put(54,-1){$\nu$-compatible}
		\put(80.5,9.3){$p$}
		\put(95.9,17.7){$q$}
		\put(80,-1){$\nu$-compatible}
	
	\end{overpic}
\caption{The $\nu$-compatibility relation.}
\label{fig:vcompatibility}
\end{figure}

\begin{definition}
A \defn{$\nu$-tree} is a maximal collection of pairwise $\nu$-compatible elements in $\A_\nu$. 
\end{definition}

We start with a lemma with many structural properties of $\nu$-trees. We omit the proof, which follows easily from the definition.
\begin{lemma}\label{lem:vtreeproperties}
Every $\nu$-tree $T$ contains the following points of $\A_\nu$:
\begin{enumerate}
\item the top-left corner of $\A_\nu$, which we call the \defn{root} of $T$;\label{it:corner}
\item the \defn{valleys} of $\nu$ (i.e.\ points between an east and a north step of $\nu$); \label{it:irrelevant1}
\item the starting points of the initial north steps of $\nu$; \label{it:irrelevant2}
\item the ending points of the final east steps of $\nu$; \label{it:irrelevant3}
\item at least one element on each column of $\A_\nu$;  \label{item:nonemptycolumn}
\item at least one element on each row of $\A_\nu$; and \label{item:nonemptyrow}
\item for every point different from the root there is either a point above it in the same column, or a point to its left in the same row, but not both. \label{it:tree-like}
\end{enumerate}
\end{lemma}

\begin{definition}
We refer to the elements of a $\nu$-tree $T$ as \defn{nodes}, and to 
the top-left corner of $\A_\nu$ as the \defn{root} of $T$.
\end{definition}

Recall that a rooted binary tree is a rooted tree in which each vertex has at most two children, which are labeled \defn{left} and \defn{right}.
We associate a rooted binary tree $\tau$ to each $\nu$-tree $T$ by connecting each element $p$ of $T$ other than the root with the next element of $T$ north or west of $p$ (exactly one of these two exists, by \eqref{it:tree-like}). See Figure~\ref{fig:Tamari35} (middle). 
The $\nu$-compatibility condition implies that the planar drawing of $\tau$ induced by $T$ is non-crossing; otherwise, the parent nodes of two crossing edges would be $\nu$-incompatible.

\begin{lemma}
Each rooted binary tree $\tau$ can be obtained uniquely as the binary tree of a $\nu$-tree~$T$, where $\nu$ is uniquely determined by $\tau$.
\end{lemma}

\begin{proof}
 We construct the inverse map that, given a rooted binary tree~$\tau$, provides a lattice path~$\nu=\nu(\tau)$ and a $\nu$-tree~$T=T(\tau)$ such that $\tau$ is the binary tree associated to $T$. The proof is recursive, and describes the points of $T$ with a coordinate system with the root at the origin, the $x$-axis pointing to the right and the $y$-axis pointing downwards.
 
 Let $\tau$ be a (non-empty) rooted binary tree. Let $\tau_\ell$ and $\tau_r$ be the left and right subtrees of the root. 
 If $\tau_\ell$ is empty, we set $\nu_\ell$ and $T_\ell$ to be empty, and $x_\ell=0$; otherwise, we set $\nu_\ell=\nu(\tau_\ell)\sfN$, $T_\ell=T(\tau_\ell)$, and $x_\ell$ to be the largest $x$-coordinate of a point in $T_\ell$. Similarly, if $\tau_r$ is empty we set $\nu_r$ and $T_r$ to be empty, and $y_r=0$; otherwise, we set $\nu_r=\sfE\nu(\tau_r)$, $T_r=T(\tau_r)$, and $y_r$ to be the largest $y$-coordinate of a point in $T_r$. 
 
 Then $\nu(\tau)$ is the concatenation $\nu=\nu_\ell\nu_r$ and $T$ is the union of $(0,0)$ with the translation of $T_\ell$ by $(0,y_r+1)$ and the translation of $T_r$ by $(x_\ell+1,0)$. Note that $T$ is a $\nu$-tree by construction. Indeed, that $T$ belongs to $A_\nu$ is direct, and the peak $\sfN\sfE$ that separates $\nu(\tau_\ell)$ and $\nu(\tau_r)$ ensures that no point of $T_\ell+(0,y_r+1)$ is $\nu$-incompatible with a point in $T_r+(x_\ell+1,0)$. 
  In Lemma~\ref{lem:numberofnodes} we will prove that all $\nu$-trees have as many nodes as the number of lattice points of $\nu$. Since this is the case for $T$, we deduce that $T$ is maximal.
  
  This construction can be described non-recursively as follows. Traversing the boundary of $\tau$ counter-clockwise from the root, there are four types of steps: a node to its left or right child, which we denote by $\downarrow$ and $\rightarrow$, respectively, and from a left or right child to its parent, which we denote by  $\uparrow$ and $\leftarrow$.  To get the lattice path~$\nu$, start with the empty path and add $\sfE$ and $\sfN$ for each $\rightarrow$ and $\uparrow$ step, respectively. To construct $T$, associate each node $v$ of $\tau$ with the point $(x,y)$ where $x$ is the number of $\rightarrow$ steps before the first appearance of $v$ in the traversal, and $y$ is the number of $\uparrow$ steps after the last appearance of $v$ in the traversal. 
\end{proof}

An example of the maps described in this proof is depicted in Figure~\ref{fig:binarytree2vtree}. Note that the path $\nu(\tau)$ is the canopy of $\tau$ as described in~\cite{PrevilleRatelleViennot2014}. This concept was first introduced in~\cite{LodayRonco1998}. 

\begin{figure}[htbp]
\begin{center}
\includegraphics[width=.85\textwidth]{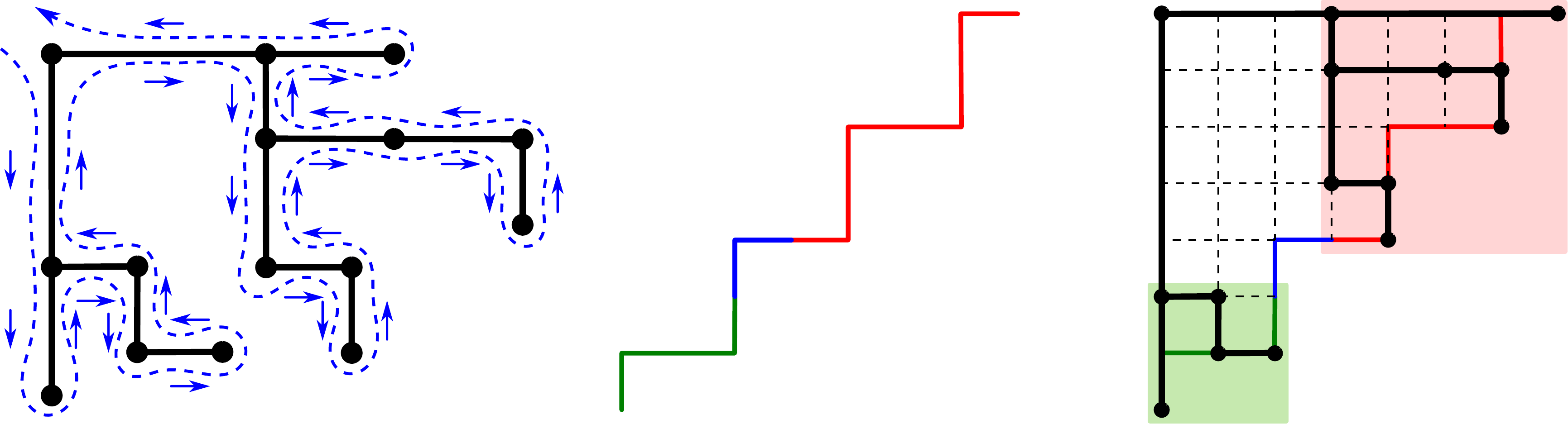}
\end{center}
\caption{The counter-clockwise traversal of the boundary of a binary tree, and its associated lattice path and $\nu$-tree.}
\label{fig:binarytree2vtree}
\end{figure}

\begin{remark}
\label{rem:existingobjects}
 The notion of $\nu$-tree has already appeared under different guises in the literature. In this paper, we use the language of $\nu$-trees on the one hand to emphasize the analogy with the binary tree representation of the classical Tamari lattice; and on the other hand because it provides structural insight on its lattice of rotations, and in particular for the definition of bracket vectors in Section~\ref{sec:brackets}.
 
 \begin{itemize}

   \item 
   In~\cite{Jonsson2005}, Jonsson defines an \defn{$\ell$-diagonal} of a polyomino $\Lambda$ as a sequence of $\ell$ boxes such that each is strictly north east of the previous one, and such that the smallest rectangle containing the boxes lies inside $\Lambda$. Hence,
   $\nu$-trees are the special case of $2$-diagonal-free maximal subsets of the Ferrers diagram~$A_{\mu}$ bounded above the path $\mu=\sfE\nu\sfN$ (the shifting with $\sfE$ and $\sfN$ is needed because in~\cite{Jonsson2005} the points are placed in the interior of the cells instead of the lattice points of the Ferrers diagram).

 \item \defn{Tree-like tableaux}, introduced in~\cite{AvalBoussicaultNadeau2013}, are fillings of a Ferrers diagram such that 
 \begin{enumerate}[(i)]
  \item\label{it:deftreelike1} the top left cell of the diagram contains a point, called the root point;
  \item\label{it:deftreelike2} for  every  non-root  pointed  cell $c$,  there  exists  either  a  pointed  cell above $c$ in the same column, or a pointed cell to its left in the same row, but not both;
  \item\label{it:deftreelike3} every column and every row possesses at least one pointed cell.
 \end{enumerate}
A \defn{crossing} of a tree-like tableau is an empty cell with both a point above it and to its left. A tree-like tableau is \defn{non-crossing} if contains no crossings.
In Proposition~\ref{prop:nutreesaretreelike} we show that $\nu$-trees are equivalent to non-crossing tree-like tableaux on the Ferrers diagram bounded by $\sfE\nu\sfN$.

Tree-like tableaux are in bijection with the widely studied permutation tableaux and alternative tableaux~\cite{CorteelWilliams2011,Postnikov2006,SteingrimssonWilliams2007}. All these families are in bijection with permutations~\cite{CorteelNadeau2009}, and each has a `Catalan' subfamily enumerated by the Catalan numbers. Non-crossing tree-like tableaux play the role of Catalan tree-like tableaux, in analogy to Catalan alternative tableaux and Catalan permutation tableaux~\cite{AvalViennot2010,viennot2007}.

\end{itemize}
\end{remark}

\begin{proposition}\label{prop:nutreesaretreelike}
 $\nu$-trees are in correspondence with non-crossing tree-like tableaux on the Ferrers diagram bounded by $\sfE\nu\sfN$.
\end{proposition}
\begin{proof}
By associating each cell with its southeast corner, we can translate between cells of the Ferrers diagram bounded by $\sfE\nu\sfN$ and lattice points of the Ferrers diagram bounded by $\nu$. Under this correspondence, properties \ref{it:deftreelike1}, \ref{it:deftreelike2}, and \ref{it:deftreelike3} defining tree-like tableaux become the $\nu$-tree properties~\eqref{it:corner}, \eqref{item:nonemptycolumn}, \eqref{item:nonemptyrow} and~\eqref{it:tree-like} from Lemma~\ref{lem:vtreeproperties}. This proves that every $\nu$-tree is a tree-like tableau. To see that it is non-crossing, note that if there was an empty cell with both a point above it and a point to its left, these points would be $\nu$-incompatible.

Conversely, the non-crossing property implies that the set of points of a non-crossing tree-like tableau are $\nu$-compatible. Indeed, if there is a couple of incompatible points, the southeast corner of the smallest rectangle containing them must be empty (by property~\eqref{it:tree-like}), and hence induces a crossing. So it only remains to prove the maximality to deduce that each non-crossing tree-like tableau is a $\nu$-tree. The number of points in a tree-like tableau is always one less than the half-perimeter of the tableau (see~\cite[pg.~5]{AvalBoussicaultNadeau2013}), which is the number of lattice points of $\nu$. In Lemma~\ref{lem:numberofnodes} we will prove that all $\nu$-trees have as many nodes as the number of lattice points of $\nu$, and hence non-crossing tree-like tableaux must be maximal.
\end{proof}

We say that two $\nu$-trees $T$ and~$T'$ are related by a \defn{right rotation} if $T'$ can be obtained by exchanging an element $q\in T$ by an element $q'\in T'$ 
as illustrated in Figure~\ref{fig:rotation}, where both $p$ and~$r$ belong to $T$ and~$T'$, and no further nodes of~$T$ or~$T'$ lie along the solid lines. The inverse operation is called a \defn{left rotation}. 

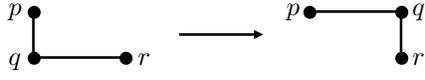
\begin{figure}[htbp]
	\begin{overpic}[width=0.4\textwidth]{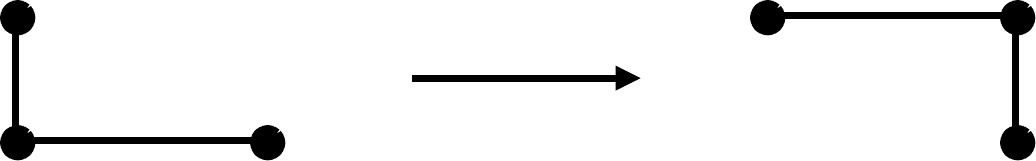}
		\put(-5,13){$p$}
		\put(-5,0){$q$}
		\put(29,0){$r$}
		\put(68,13){$p$}
		\put(101,13){$q'$}
		\put(101,0){$r$}
	\end{overpic}
\caption{Right rotation operation on $\nu$-trees.}
\label{fig:rotation}
\end{figure}

\begin{definition}
 The \defn{rotation poset} of $\nu$-trees is the partial order on the set of $\nu$-trees defined by the covering relations $T<T'$ whenever $T'$ is obtained from $T$ by a right rotation. 
\end{definition}

\begin{theorem}
\label{thm:rotationlattice}
The rotation poset of $\nu$-trees is a lattice. 
\end{theorem}
 
We will give two proofs of this result in Sections~\ref{sec:vTamari_rotationlattice} and~\ref{sec:brackets}. 

\begin{example}[Complete binary trees]
For the path $\nu=(\sfN\sfE)^n$, $\nu$-trees coincide with complete binary trees with $n$ internal nodes, as illustrated in Figure~\ref{fig:binarytrees}. The rotation coincides with the usual rotation on complete binary trees. The rotation lattice of $\nu$-trees is therefore the classical Tamari lattice.

\begin{figure}[htbp]
\includegraphics[width= 0.8\textwidth]{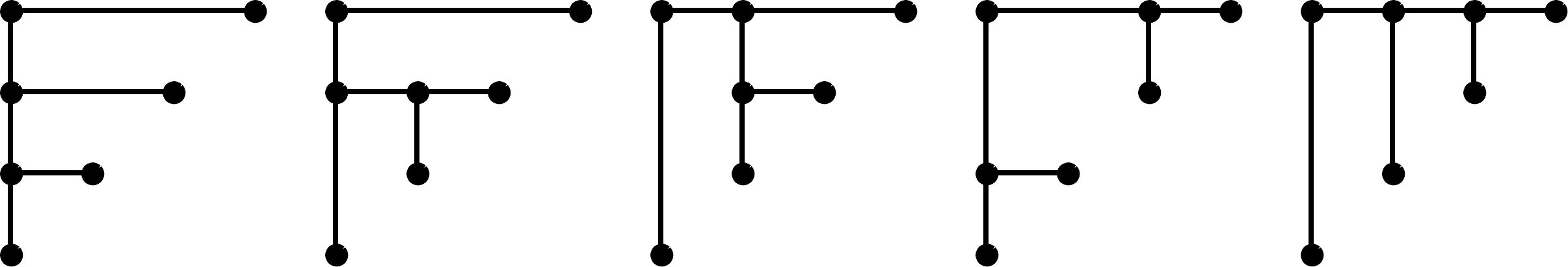}
\caption{Complete binary trees seen as $\nu$-trees for $\nu=(\sfN\sfE)^n$.}
\label{fig:binarytrees}
\end{figure}

\end{example}

We present now some properties of $\nu$-tree rotations that will be useful later. For this, we define the \defn{minimal $\nu$-tree} $\Tmin$ as the subset of $\A_\nu$ containing all the points on the left most column, together with all ending points of the east steps of~$\nu$. The \defn{maximal  $\nu$-tree} $\Tmax$ is the subset of $\A_\nu$ containing all the points on the top most row, and all the starting points of the north steps of $\nu$. These are clearly $\nu$-trees, with the property that $\Tmin$ admits no left rotation and $\Tmax$ admits no right rotation; they are shown in Figure~\ref{fig:Tamari35} (middle) as the minimal and maximal elements of the lattice.

The following lemma is the special case of $k=1$ in~\cite[Lemma~3.3]{rubey_maximal_2012}, although without a proof there. 
\begin{lemma}
A rotation of a $\nu$-tree is also a $\nu$-tree. 
\end{lemma}

\begin{proof}
Let $T$ be a $\nu$-tree and $T'=T\ssm\{q\}\cup \{q'\}$ be obtained from $T$ by a right rotation involving $p,q,r\in T$, as in Figure~\ref{fig:rotation}.

{\it Rotations preserve compatibility:}
Assume there is some $s\in T\cap T'$ such that $s\ncomp q'$. If $s$ is due northeast of $q'$ then $s\ncomp p$ in~$T$. 
If~$s$ is due southwest of $q'$, there are three cases to consider. \begin{enumerate*}[label=(\roman*)] \item Since $s,q\in T$ are $\nu$-compatible,~$s$ cannot lie due northeast or southwest of $q$; \item $s$ cannot lie due northwest of $q$, for then $s\ncomp p$; \item finally, if $s$ is due southeast of $q$, then $s\ncomp r$. \end{enumerate*}
Either way we get a contradiction, so $T'$ consists of pairwise $\nu$-compatible points. The proof for left rotation is similar.

{\it Rotations preserve maximality:}
Assume $T'$ is not maximal, so there is some $s\in \A_\nu$ with $s\notin T'$ that is $\nu$-compatible with every element in $T'$. \begin{enumerate*}[label=(\roman*)] \item If $s$ is not in the rectangle with vertices $p,q,r,q'$, we can obtain $T\cup s$ by applying a left rotation to $T'\cup s$. 
 \item If $s$ lies on the rectangle with vertices $p,q,r,q'$, it must necessarily lie on the same row or column of $q'$, as otherwise $s\ncomp q'$. 
Say that $s$ lies on the same row as $q'$ (the other case being analogous). Then we can obtain $T\cup s'$ from $T'\cup s$ by applying two left-rotations as in Figure~\ref{fig:maximality}.\end{enumerate*} In both situations we get a contradiction on the maximality of $T$ because left rotations preserve $\nu$-compatibility.

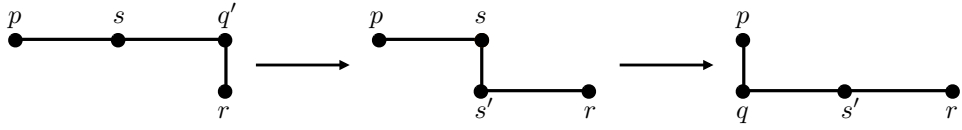
\begin{figure}[htbp]
	\begin{overpic}[width=\textwidth]{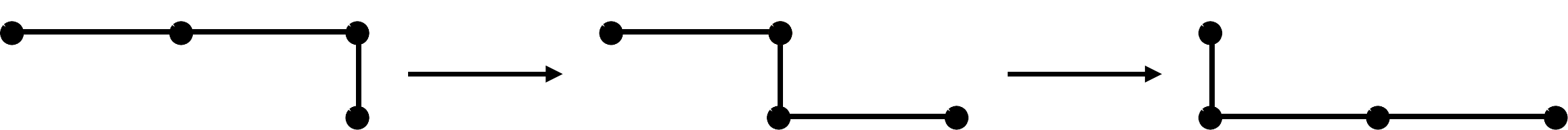}
		\put(0,8){$p$}
		\put(11,8){$s$}
		\put(22,8){$q'$}
		\put(22,-2){$r$}
		\put(38,8){$p$}
		\put(49,8){$s$}
		\put(49,-2){$s'$}
		\put(60.5,-2){$r$}
		\put(76.5,8){$p$}
		\put(76.5,-2){$q$}
		\put(87.5,-2){$s'$}
		\put(98.5,-2){$r$}
	\end{overpic}
\caption{Producing $T\cup s'$ from $T'\cup s$ via two left rotations.}
\label{fig:maximality}
\end{figure}
\end{proof}

The first statement of the following lemma is the special case $k=1$ of~\cite[Theorem~3.8]{rubey_maximal_2012}. The second statement concerning the cardinality of $\nu$-trees follows from~\cite[Theorem~10]{Jonsson2005}.

\begin{lemma}\label{lem:numberofnodes}
The rotation poset of $\nu$-trees is connected. In particular, all $\nu$-trees have the same number of nodes, which equals the number of lattice points on $\nu$. 
\end{lemma}   

\begin{proof}
Note that if $T$ contains all the points on the top row of $\A_\nu$ then $T=\Tmax$. Assume that $T\neq \Tmax$ and let $j$ be the first column whose highest point does not belong to $T$ (we index the columns from left to right, and the rows from bottom to top). Note that $j\geq 2$ because the top-left corner (i.e. the root of $T$) belongs to every $\nu$-tree (cf. Lemma~\ref{lem:vtreeproperties}). Let~$p'\in T$ be the point in column $j-1$ on the topmost row, and $r=(i,j)\in T$ be the highest point of $T$ in column $j$ (which is non-empty by item~(\ref{item:nonemptycolumn}) in Lemma~\ref{lem:vtreeproperties}). We claim that that $q=(i,j-1)\in T$. Indeed, assume there is a point $t\in T$ such that $t\ncomp q$. If $t$ lies due southwest of $q$, then $t\ncomp p'$ as well. If $t$ lies due northeast of $q$, it must lie due northeast of $r$ too, and this would mean that $t\ncomp r$. Both cases yield contradictions.

Let $p$ be the next point of $T$ due north of $q$ (which may equal $p'$). Then the points $p,q,r\in T$ are in the situation of Figure~\ref{fig:rotation}, so we may right-rotate $q$ to a point $q'$ due north of $r$. Thus, it is always possible to apply a right rotation to a non-maximal tree. Since right rotation is an acyclic operation and the number of $\nu$-trees is finite, we eventually reach $\Tmax$ by a finite sequence of right rotations. The number of nodes on $\Tmax$ is clearly equal to the number of lattice points on~$\nu$.
\end{proof}

\begin{lemma}\label{lem:flipsequalrotations}
Two $\nu$-trees differ by a single element if and only if they are related by a rotation.
\end{lemma}

\begin{proof}
The ``{if}'' direction holds by definition. For the ``{only if}'' direction, consider two $\nu$-trees $T$ and~$T'$ such that $T'=T\ssm q \cup q'$. It is clear that $q\ncomp q'$, for otherwise $T\cup q'$ would consist of pairwise $\nu$-compatible points, contradicting the maximality of $T$. We claim that the points $p$ and $r$, lying respectively on the northwest and southeast corners of the smallest rectangle containing $q,q'$, belong to both $T$ and~$T'$, and that no further points of $T$ or $T'$ lie on that rectangle. 

Indeed, one observes that the existence of any point $s\in T\cap T'$ with $s\ncomp p$ or $s\ncomp r$ would imply incompatibilities $s\ncomp q$ or $s\ncomp q'$. We leave the easy details to the reader. Hence, $p$ and $r$ must belong to $T$ and $T'$ by maximality.  Moreover, the rectangle $pqrq'$ must be empty because any point inside would be incompatible with $q$ or $q'$. From this it follows that $T$ and $T'$ are related by a rotation. 
\end{proof}

\section{The \texorpdfstring{$\nu$}{v}-Tamari lattice as a rotation lattice}
\label{sec:vTamari_rotationlattice}

In this section we show that the rotation lattice of $\nu$-trees is equivalent to the $\nu$-Tamari lattice of Pr\'eville-Ratelle and Viennot~\cite{PrevilleRatelleViennot2014}.

\subsection{\texorpdfstring{$\nu$}{v}-Tamari lattices}
\label{sec:vTamariLattices}
We identify \defn{lattice paths} that consist of a finite number of north and east unit steps with words on the alphabet $\{\sfN,\sfE\}$. 
Given a lattice path~$\nu$, a \defn{$\nu$-path} is a lattice path with the same endpoints as $\nu$ that is weakly above~$\nu$. The set of $\nu$-paths is endowed with a partial order which we now recall. 

\begin{definition}
The \defn{$\nu$-Tamari poset} \defn{$\Tam{\nu}$} on the set of $\nu$-paths is the transitive closure~$<_\nu$ of the covering relation~$\lessdot_\nu$ defined as follows:

Let $\mu$ be a $\nu$-path. For a lattice point $p$ on $\mu$ define the distance \defn{$\horiz_\nu(p)$} to be the maximum number of horizontal steps that can be added to the right of $p$ without crossing $\nu$. Given a \defn{valley} $p$ of $\mu$ (a point preceded by an east step $\sfE$ and followed by a north step $\sfN$) we let $q$ be the first lattice point in $\mu$ after $p$ such that $\horiz_\nu(q)=\horiz_\nu(p)$. We denote by $\mu_{[p,q]}$ the subpath of $\mu$ that starts at~$p$ and finishes at $q$, and consider the path $\mu'$ obtained from $\mu$ by switching $\sfE$ and $\mu_{[p,q]}$. 
The covering relation is defined to be $\mu \lessdot_\nu \mu'$.

\end{definition}

An example is illustrated in Figure~\ref{fig:Tamari35} (left). The case $\nu=(\sfN\sfE)^n$ yields the classical Tamari lattice.

In~\cite{PrevilleRatelleViennot2014}, Pr\'eville-Ratelle and Viennot proved several structural results about~$\Tam{\nu}$. In particular, they showed that it has the structure of a lattice.  

\begin{theorem}[Pr\'eville-Ratelle and Viennot {\cite{PrevilleRatelleViennot2014}}]
\label{thm:PRVlattice}
The $\nu$-Tamari poset is a lattice. 
\end{theorem}

\subsection{The rotation lattice of \texorpdfstring{$\nu$}{v}-trees is isomorphic to the \texorpdfstring{$\nu$}{v}-Tamari lattice}\label{sec:RotationlatticeTamarilatticeIsomorphism}
We present now a bijection that induces an isomorphism between the rotation lattice of $\nu$-trees and the $\nu$-Tamari lattice.
Theorem~\ref{thm:rotationlattice} is then a direct consequence of Theorem~\ref{thm:PRVlattice}.

\begin{theorem}\label{thm:rotationlatticetrees}
The $\nu$-Tamari lattice is isomorphic to the rotation lattice of $\nu$-trees.
\end{theorem}

To describe the isomorphism, consider the following maps between $\nu$-paths and $\nu$-trees.
Let $\mu$ be a $\nu$-path.
We construct a $\nu$-tree $T=\rflush(\mu)$ by \defn{$\rflush$ight-flushing} the points in $\mu$ row-wise in the following way.
First label the points in $\mu$ in the order they appear along the path, traversed from southwest to northeast.
Starting from the bottom row and proceeding upwards, the points in a row are placed as rightmost as possible on the same row of $\A_\nu$, in the assigned order and avoiding $x$-coordinates that are \emph{forbidden} by previous flushed rows. The $x$-coordinates forbidden by a row are those of its right-flushed lattice points, excepting the last (i.e. leftmost) one. The collection of lattice points obtained by right flushing all the points in~$\mu$ constitutes the $\nu$-tree $\rflush(\mu)$. This is illustrated in Figure~\ref{fig:justification} (top). 

Symmetrically, the inverse map is defined as a row-wise \defn{$\lflush$eft-flushing} of the lattice points in a $\nu$-tree $T$. First we label the points of $T$ in the order they appear when traversed from bottom to top and from right to left. Starting from the bottom row and proceeding upwards, the points in a row are placed as leftmost as possible on the same row, in the assigned order and avoiding $x$-coordinates that are \emph{forbidden} by previous flushed rows. This time, the $x$-coordinates forbidden by a row are those of its left-flushed lattice points, excepting the last (i.e. rightmost) one. The resulting collection of lattice points forms the path $\mu=\lflush(T)$; see Figure~\ref{fig:justification} (bottom). 

Thus, the above maps give a one-to-one correspondence between lattice points in a $\nu$-path $\mu$ and nodes in a $\nu$-tree $T=\rflush(\mu)$. We will often decorate labels of lattice points in a $\nu$-tree with an overline, and write $p\leftrightarrow \ol p$ for this correspondence\footnote{We will be rather lax with the overlined notation for nodes of $\nu$-trees, and only employ it when the correspondence $p\leftrightarrow \ol p$ induced by flushing needs to be stressed.}. We show now that $\rflush$ and $\lflush$ are indeed well-defined inverse bijections.

\begin{figure}[htpb]
\centering 
\includegraphics[width=0.8\textwidth]{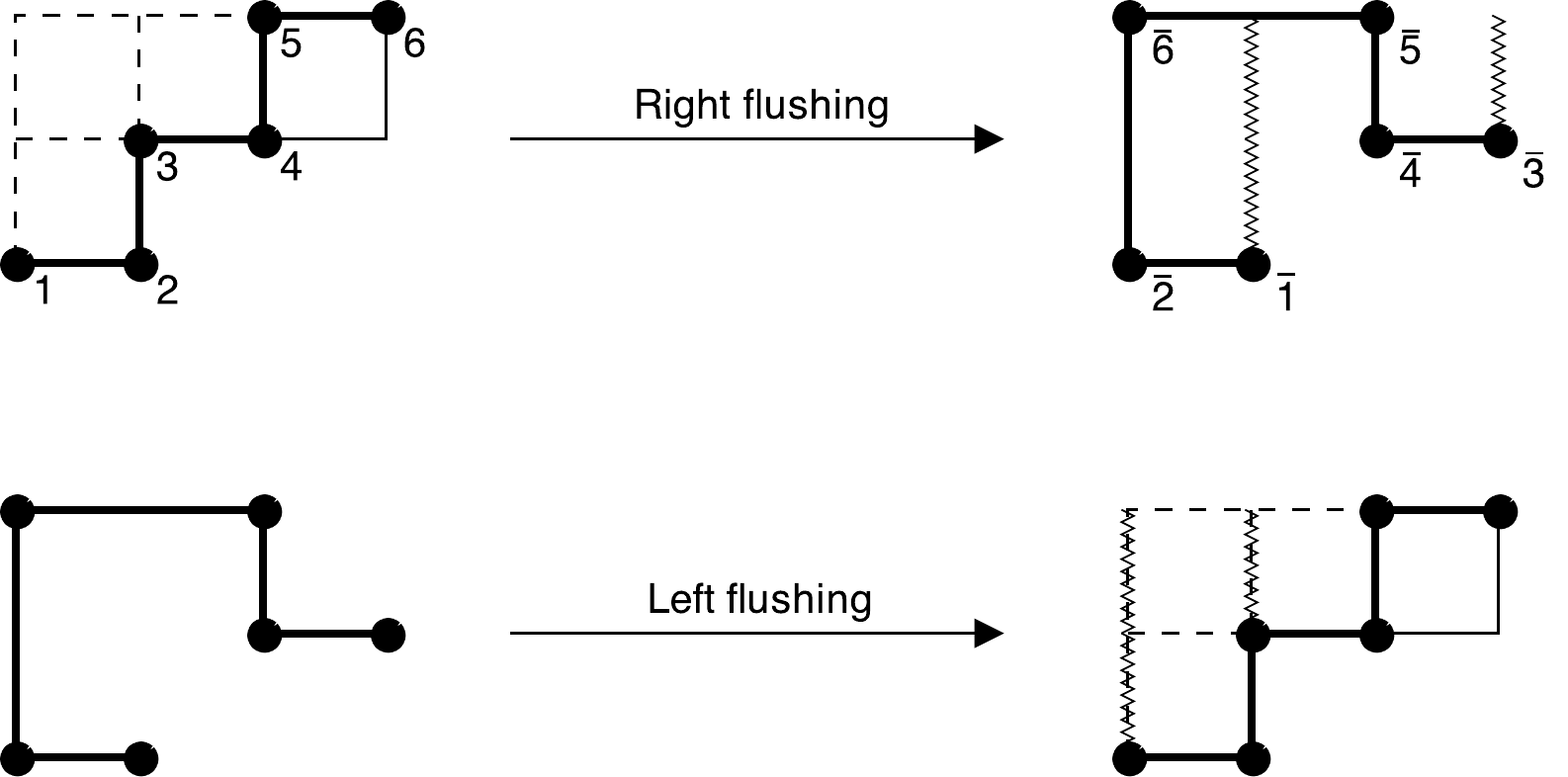}
 \caption{Bijection between $\nu$-paths and $\nu$-trees. Forbidden $x$-coordinates are indicated by vertical creased lines.}
 \label{fig:justification}
\end{figure}

\begin{proposition}\leavevmode
\label{prop_flushingbijection}
\begin{enumerate}[leftmargin=*, ref=\emph{(\arabic*)}]
 \item \label{it:welldefn} The right and left flushing maps $\rflush,\lflush$ are well-defined.
 \item \label{it:bijections} $\rflush$ and $\lflush$ are inverse bijective correspondences between $\nu$-paths and $\nu$-trees. 
 \item \label{it:covertorotation} Two $\nu$-paths $\mu,\mu'$ are related by a $\nu$-Tamari covering relation if and only if the $\nu$-trees $\rflush(\mu),\rflush(\mu')$ are related by rotation.
\end{enumerate}
\label{prop:Tamari_rotationtrees}
\end{proposition}

\begin{proof}
\ref{it:welldefn}: First, we verify that right-flushing can proceed up to the top row, in the sense that there are always $x$-coordinates available to place a flushed point. Given a point $p$ on a $\nu$-path $\mu$ and the corresponding right-flushed point $\ol p$, this means that the difference between the width of the Ferrers diagram $\F_\nu$ at the row on which $\ol p$ lies and the number of $x$-coordinates forbidden prior to $\ol p$ is nonnegative. Indeed, we recognize the subtrahend as the number of east steps of $\mu$ prior to $p$, so the difference equals the quantity $\horiz_\nu(p)$ (cf. Section~\ref{sec:vTamariLattices}), which is nonnegative by construction.

Likewise, to left-flush a node $\ol p$ of a $\nu$-tree $T$ to a point $p$ via $\lflush$, we require again that the difference between the width of the Ferrers diagram $\F_\nu$ at the row on which $p$ lies and the number of $x$-coordinates forbidden prior to $p$ be nonnegative. We recognize this quantity as the number of horizontal edges in the unique path in $T$ from $\ol p$ to the root, and denote it by $\hroot_T(\ol p)$. Clearly $\hroot_T(\ol p)\geq 0$ for every $\ol p\in T$

We now check that $\rflush$ and $\lflush$ map to $\nu$-trees and $\nu$-paths, respectively. Given a $\nu$-path $\mu$, we claim that $T:=\rflush(\mu)$ is a $\nu$-tree. 
It is not difficult to see that, by construction of $\rflush$, the points of $T$ are pairwise $\nu$-compatible. 
Indeed, if $\overline p$ and $\overline q$ are incompatible, with $\overline q$ northeast of $\overline p$, then either there is a point in $\overline p$'s row in the same column as $\overline q$, or this $x$-coordinate was already forbidden by a previous row; in both cases, the $x$-coordinate of $\overline q$ is forbidden, which is a contradiction.
On the other hand, $T$ is maximal because, by Lemma~\ref{lem:numberofnodes}, the number of nodes in every $\nu$-tree equals the number of lattice points in $\nu$, which in turn equals the number of nodes in $\mu$.

Given a $\nu$-tree $T$, we claim that $\mu:=\lflush(T)$ is a $\nu$-path. Indeed, the correspondence $p\leftrightarrow \ol p$ between lattice points of $\mu$ and nodes of $T$ induces the equality $\horiz_\nu(p) = \hroot_T(\overline p)$, since the expressions for these quantities as differences agree, modulo exchanging $p$ and $\ol p$. Since $\hroot_T(\overline p)$ is non-negative for every $\ol p\in T$,  $\mu$ lies weakly above $\nu$.

\ref{it:bijections}: Injectivity of the right flushing map $\mu\mapsto \rflush(\mu)$ follows  because $\rflush(\mu)$ depends only on the number of lattice points of $\mu$ on each row, and this statistic uniquely determines the path. 
The surjectivity of $\rflush$ follows from the left flushing map $T\mapsto \lflush(T)$, which is the inverse of $\rflush$.

 \ref{it:covertorotation}: It remains to show that the covering relation on $\nu$-paths translates to rotation on $\nu$-trees, and vice versa.
 Let $p$ be a valley of $\mu$, $q$ be the first point on $\mu$ after $p$ such that $\horiz_\nu(q)=\horiz_\nu(p)$, and $\mu'$ be the path obtained from $\mu$ by switching the east step preceding $p$ and $\mu_{[p,q]}$. 
 Let $T=\rflush(\mu)$ and $\overline p,\overline q$ be the corresponding nodes of $p$ and $q$, respectively. We claim that $\overline q$ is the parent of $\overline p$ in $T$. The reason is that, since $\overline p$ is the leftmost node of its row, all the nodes with labels between $\overline p$ and its parent have larger horizontal distance to the root. Therefore, the parent $\overline q$ of $\overline p$ is the first node after $\overline p$ such that $\hroot_T(\overline q) = \hroot_T(\overline p)$.  
 The tree $T'=\rflush(\mu')$ is then obtained from $T$ by replacing $\overline p$ by a node on the same row as $\overline q$. Therefore $T'$ is necessarily equal to rotating $\overline p$ to the right in $T$.
 The same argument works to prove the converse correspondence.
\end{proof}

Theorem~\ref{thm:rotationlatticetrees} is a direct consequence of Proposition~\ref{prop:Tamari_rotationtrees}.
As a corollary, we get an alternative proof of one of the main results in~\cite{PrevilleRatelleViennot2014}. 

\begin{corollary}[{\cite[Theorem~2]{PrevilleRatelleViennot2014}}]\label{cor:combinatorialduality}
Let $\reverse \nu$ be the path obtained by reading $\nu$ backwards and replacing the east steps by north steps and vice versa. Then $\Tam{\nu}$ is isomorphic to the dual lattice of $\Tam{\reverse\nu}$.
\end{corollary} 

\begin{proof}
The $\reverse\nu$-trees can be obtained from $\nu$-trees by reflecting them on the line of slope minus 1 passing through the root. This reflecting operation, that we represent by~$T\rightarrow \reverse T$, maps $\Tam{\nu}$ to the dual of $\Tam{\reverse \nu}$ bijectively, because it turns right rotations on $T$ into left rotations on $\reverse T$.
\end{proof}

\begin{remark}
The correspondence between $\Tam{\nu}$ and the dual of $\Tam{\reverse\nu}$ can be made explicit at the level of lattice paths by the bijection 
$\lflush\ \circ \reverse{(\cdot)} \circ\ \rflush$.
We illustrate this composition in Figure~\ref{fig:justifiedpath}.
\end{remark}

\begin{figure}[htbp]
\begin{overpic}[width=\textwidth]{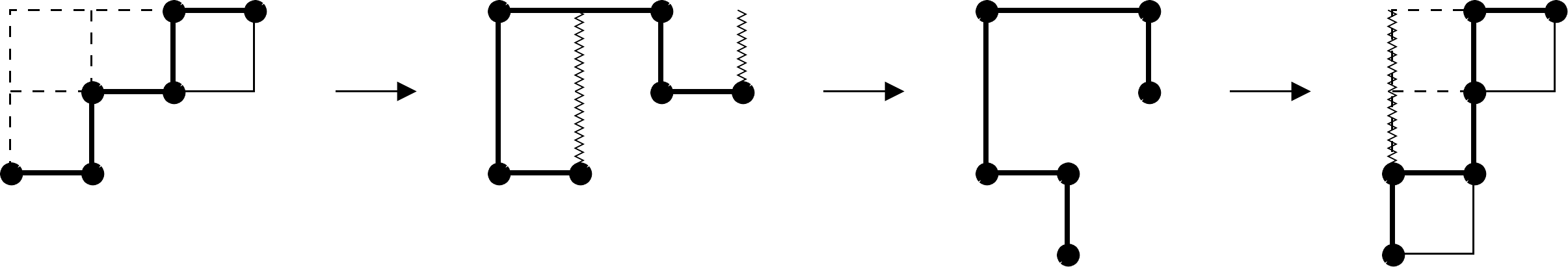}
 \put(22,12){{ $\rflush$}}
 \put(53,12){{ $\leftarrow$}}
 \put(79,12){{ $\lflush$}}
\end{overpic}
\caption{Bijection between $\nu$-paths and $\protect\reverse{\nu}$-Dyck paths via flushing.}
\label{fig:justifiedpath}
\end{figure}

\begin{remark}
In~\cite{tams2017}, the concept of $(I,J)$-trees was introduced in order to produce geometric realizations of $\nu$-Tamari lattices. The $\nu$-trees presented in this paper are equivalent to the \emph{grid representation} of the $(I,J)$-trees (cf.\ \cite[Remark~2.2]{tams2017}). 
\end{remark}

\section{The rotation lattice via bracket vectors}

\label{sec:brackets}

In this section, we provide a direct proof of the lattice property for the rotation lattice of $\nu$-trees. The core notion is that of a bracket vector, which has a natural meaning in the graph theoretical context of $\nu$-trees. For completeness, we also include a description of bracket vectors in terms of lattice paths.

\subsection{\texorpdfstring{$\nu$}{v}-bracket vectors}

\begin{definition}\label{def:bracketvectors}
Let $\nu$ be a lattice path from $(0,0)$ to $(m,n)$ with length $\ell = \ell(\nu)=m+n$. The \defn{minimal $\nu$-bracket vector} $\bmin$ is a vector consisting of $\ell+1$ non-negative integers obtained by reading the $y$-coordinates of the lattice points on $\nu$ in the order they appear along the path. We define the set of \defn{fixed positions} as the set $F=\{f_0,f_1,\dots,f_n\}$ where $f_k$ is the position of the last appearance of~$k$ in $\bmin$. 
A \defn{$\nu$-bracket vector} is a vector $\sfb = (b_1,\dots, b_{\ell+1})$ satisfying the following properties:
\begin{enumerate}
\item $b_{f_k}=k$ for $0\leq k \leq n$; \label{def:bracketvectors1}
\item $\bmin_i \leq b_i \leq n$ for all $i$; \label{def:bracketvectors2}
\item the sequence $(b_1,\dots, b_{\ell+1})$ is $121$-avoiding. \label{def:bracketvectors3_alternative}
\end{enumerate}
Recall that a sequence is $121$-avoiding if it does not contain any subsequence~$(k,k',k)$ with $k<k'$.
Condition~(\ref{def:bracketvectors3_alternative}) in this definition can be equivalently replaced by 
\begin{enumerate}[label={(\arabic*$'$)}, ref={\arabic*$'$}, start=3]
\item if $b_i=k$, then $b_j\leq k$ for $i\leq j \leq f_k$. \label{def:bracketvectors3}.
\end{enumerate}
\end{definition}

\begin{theorem}\label{thm:bracketlattice}
The $\nu$-Tamari lattice is isomorphic to the lattice of $\nu$-bracket vectors under componentwise order.
\end{theorem}

This theorem provides a simple description of the lattice. Its proof is delayed until Section~\ref{sec:proofbracketlattice}, and the description of the meet and join operations are presented in Section~\ref{sec:meetjoin}.

\begin{remark}
Our definition of bracket vectors is inspired by the work of Huang and Tamari in~\cite{HuangTamari1972}, who introduced a notion of \emph{right bracketings} of a word $x_0x_1\dots x_n$ to provide a simple proof of the lattice property for the classical Tamari lattice $\Tam{n}$. Each right bracketing is encoded by an $n$-vector satisfying similar properties as in our definition of $\nu$-bracket vectors (Definition~\ref{def:bracketvectors}). Indeed, their vectors can be obtained from our $\nu$-brackets vectors, for $\nu=(\sfN\sfE)^n$, by removing the values at the fixed positions $f_0,\dots, f_n$. Our Theorem~\ref{thm:bracketlattice} generalizes the main result in~\cite{HuangTamari1972}. 
\end{remark}

\begin{remark}
Pr\'eville-Ratelle and Viennot showed that the $\nu$-Tamari lattice can be obtained as an interval in the classical Tamari lattice~\cite{PrevilleRatelleViennot2014}. Restricting the classical bracket vectors of Huang and Tamari in~\cite{HuangTamari1972} to this interval gives a similar description of the $\nu$-Tamari lattice as in Theorem~\ref{thm:bracketlattice}. 
However, the description of $\nu$-bracket vectors in this paper is simpler and more direct; it also uncovers essential information about $\nu$-trees, as we will now see. 
\end{remark}

\subsection{Bracket vectors from \texorpdfstring{$\nu$}{v}-trees}
The bracket vector of a $\nu$-tree is obtained using the in-order traversal\footnote{The in-order is called \defn{symmetric order} by Pr\'eville-Ratelle and Viennot in~\cite{PrevilleRatelleViennot2014}} of the tree, which is defined recursively as follows: if $x$ is the root, $A$ is its left subtree and $B$ its right subtree, we visit the nodes $A$ in in-order, then visit $x$ and finally visit $B$ in in-order.

\begin{definition}[Bracket vector of a $\nu$-tree]
We label each node of a $\nu$-tree $T$ by its $y$-coordinate. The \defn{bracket vector} $\sfb(T)$ is the result of reading the labels of the nodes in \defn{in-order}. Note that $\bmin=\sfb(\Tmin)$, the bracket vector of the minimal $\nu$-tree. 
\end{definition}

An example of the bracket vectors for all $\sfE\sfN\sfE\sfE\sfN$-trees is illustrated in Figure~\ref{fig:Tamari35} (right), where the bold numbers are the values at the fixed positions $(f_0,f_1,f_2)=(2,5,6)$. The rotation operation induces a simple operation at the level of bracket vectors, schematically depicted in Figure~\ref{fig:bracket_rotation}. 
Let $T'$ be a rotation of $T$. If $\sfb(T)=AxByC$, where $x$ is the entry corresponding to the node that is being rotated, and $y$ that of its parent; then $\sfb(T')=AyByC$.
Note that in such rotation all entries of the vectors remain unchanged except for the label corresponding to the rotated node. 
This explains why the values at the fixed positions remain unchanged for all~$\nu$-trees.

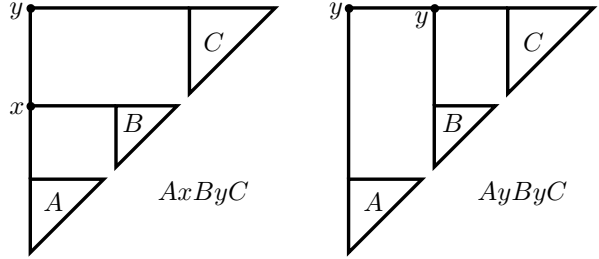
\begin{figure}[htpb]
\centering 
	\begin{overpic}[width=0.6\textwidth]{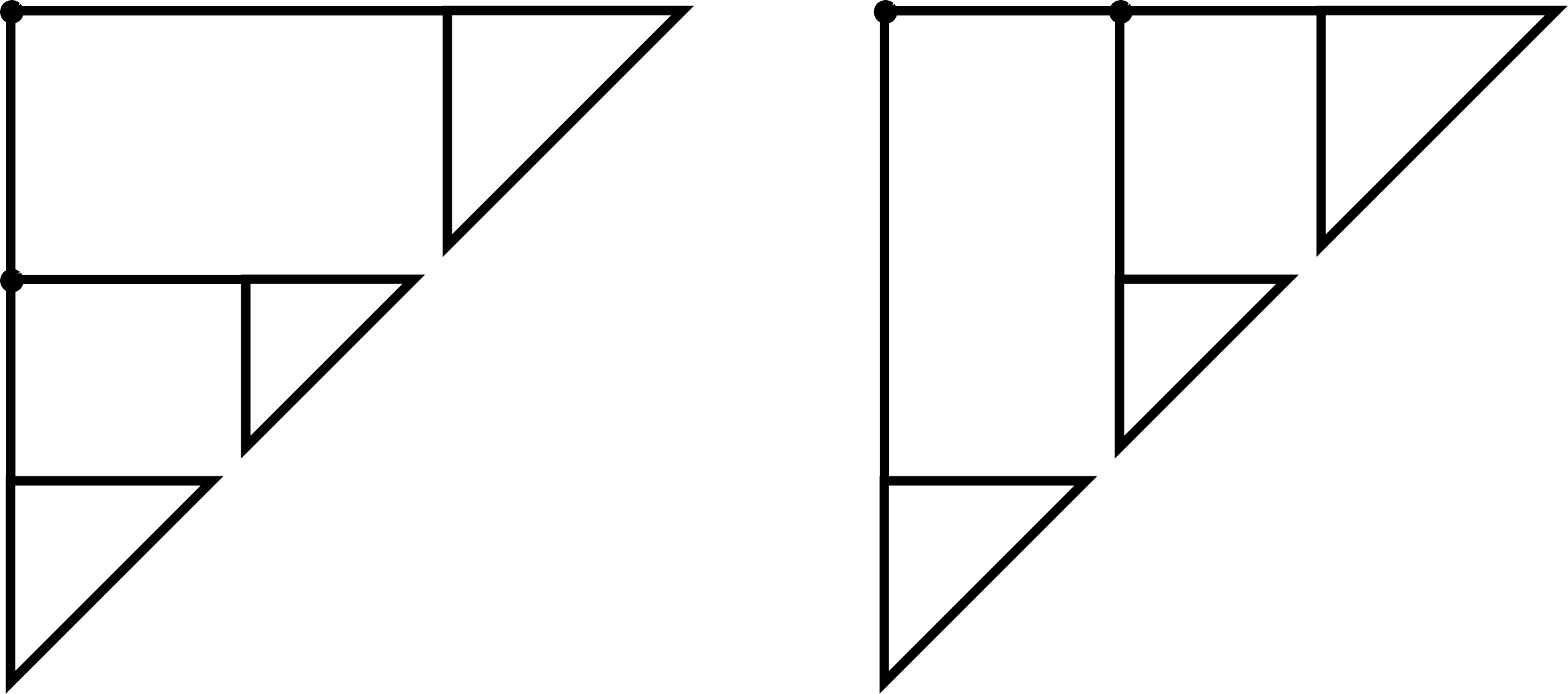}
		\put(3,8){$A$}
		\put(16.8,21.8){$B$}		
		\put(31,36){$C$}
		\put(-3,25){$x$}
		\put(-3,42.5){$y$}
		\put(23,10){$AxByC$}
		\put(59,8){$A$}
		\put(72.8,21.8){$B$}		
		\put(87,36){$C$}
		\put(68,41){$y$}
		\put(53,42.5){$y$}
		\put(79,10){$AyByC$}		
	\end{overpic}
 \caption{The bracket rotation.}
 \label{fig:bracket_rotation}
\end{figure}

\begin{remark}
Our notion of bracket vector is closely related 
to other definitions of bracket vectors in the literature. For instance, in~\cite{BjornerWachs1997} Bj\"orner and Wachs define (after Knuth~\cite{Knuth93} and Pallo~\cite{Pallo86}) the bracket vector of a complete binary tree $T$ on $\ell+2$ leaves as the sequence $r(T)=(r_1,\ldots,r_{\ell+1})$, where $r_i$ denotes the number of internal nodes in the right subtree of the internal node at position $i$ in the in-order traversal of $T$. In terms of $r(T)$, the entries of the bracket vector $\sfb(T)=(b_1,\ldots,b_{2\ell+3})$ can be recovered as $b_{2i-1}=i-1$ for $1\leq i \leq \ell+2$ (fixed positions) and $b_{2i}=r_i+i$ for $1\leq i \leq \ell+1$. The latter follows because $b_{2i}$ equals the number of leaves weakly preceding internal node $2i$ in the \emph{post-order} traversal\footnote{In the post-order traversal of a binary tree, if $x$ is the root, $A$ is its left subtree and $B$ its right subtree, we visit the nodes $A$ in post-order, then visit $B$ in post-order, and finally visit $x$.} of $T$ minus one (see~\cite[Remark 2.2]{tams2017}), which in turn equals $r_i+i$. 
One can generalize the definition of $r(T)$ for more general $\nu$-trees, and its relation with $\sfb(T)$ is analogous.  
\end{remark}

\subsection{Bracket vectors from \texorpdfstring{$\nu$}{v}-paths}
Bracket vectors can also be easily defined in terms of $\nu$-paths. 

\begin{definition}[Bracket vector of a $\nu$-path]
We label each lattice point of a $\nu$-path $\mu$ by its $y$-coordinate. 
The \defn{bracket vector} $\sfb(\mu)$ is constructed from the labels as follows. We start with an empty vector of length $\ell+1$ and start filling its entries. For $k$ varying from $0$ to $n$, 
we set as many entries of the vector equal to $k$ as there are lattice points in row $k$, rightmost possible but before the fixed position $f_k$.
Note that $\bmin=\sfb(\nu)$, the bracket vector of $\nu$ itself. 
\end{definition}

An example is illustrated in Figure~\ref{ex:bracketvector_from_path}. The underlined numbers denote the values at the fixed positions $(f_0,f_1,f_2,f_3,f_4)=(3,4,7,8,10)$.

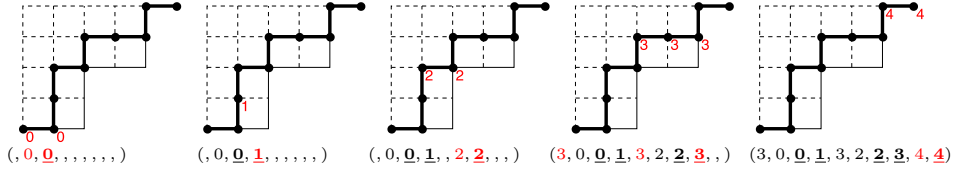
\begin{figure}[htbp]
\begin{center}
	\begin{overpic}[width=\textwidth]{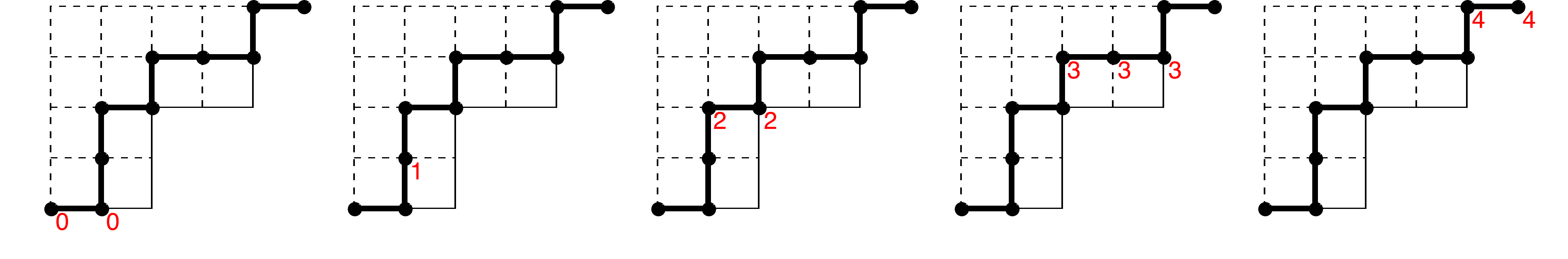}
	\put(1.5,0){\tiny$(,{\color{red}0},{\color{red}\underline{\bf0}},,,,,,,)$}
	\put(21.5,0){\tiny$(,0,\underline{\bf0},{\color{red}\underline{\bf1}},,,,,,)$}
	\put(39.5,0){\tiny$(,0,\underline{\bf0},\underline{\bf1},,{\color{red}2},{\color{red}\underline{\bf2}},,,)$}
	\put(58.5,0){\tiny$({\color{red}3},0,\underline{\bf0},\underline{\bf1},{\color{red}3},2,\underline{\bf2},{\color{red}\underline{\bf3}},,)$}
	\put(79.5,0){\tiny$(3,0,\underline{\bf0},\underline{\bf1},3,2,\underline{\bf2},\underline{\bf3},{\color{red}4},{\color{red}\underline{\bf4}})$}
	\end{overpic}
\caption{Bracket vectors from $\nu$-paths.}
\label{ex:bracketvector_from_path}
\end{center}
\end{figure}

\begin{proposition}\label{prop:bracketvectors}
The bracket vectors for $\nu$-trees and $\nu$-paths are characterized by Definition~\ref{def:bracketvectors}. Moreover, if $T=\rflush(\mu)$ is the $\nu$-tree corresponding to a $\nu$-path $\mu$ under the right flushing bijection, then $\sfb(T)=\sfb(\mu)$.
\end{proposition}

\begin{proof}
The result follows from the following three observations:

\begin{enumerate}[label=\emph{(\roman*)},labelindent=0pt, wide, labelwidth=!]
\item \emph{The bracket vector of a $\nu$-tree $T$ satisfies the properties in Definition~\ref{def:bracketvectors}}. 
Note that $\sfb(\Tmin)$ satisfies Property~(\ref{def:bracketvectors1}). Since rotations do not change the value at position $f_k$ in a bracket vector, then $b_{f_k}=k$ for every tree.   

Property~(\ref{def:bracketvectors2}) follows from the fact that each right rotation increases the values of the bracket vector. Property~(\ref{def:bracketvectors3}) follows because, between two values $k$ in $\sfb(T)$, we read in the in-order some labels of nodes that are descendants of the node with the first value~$k$. These labels are all less than or equal to $k$.

\item \label{it:uniquebvector}\emph{Each $\nu$-bracket vector can be obtained uniquely as the bracket vector of a $\nu$-path~$\mu$}. 
Let $\sfb$ be a $\nu$-bracket vector (as in Definition~\ref{def:bracketvectors}) and $\mu$ be the unique path containing as many lattice points on row $k$ as values $k$ in $\sfb$. Since all values $\leq k$ in $\sfb$ are placed at positions $\leq f_k$, the path $\mu$ is weakly above $\nu$ and therefore is a $\nu$-path. If $\mu'$ is a $\nu$-path with $\mu'\neq \mu$, then $\mu,\mu'$ have a different number of points on some row $k$. But then $\sfb\neq \sfb(\mu')$ because they do not have the same number of instances of $k$, so the map $\mu\mapsto \sfb(\mu)$ is injective. To check that $\sfb=\sfb(\mu)$, note that~$\sfb$ can be uniquely reconstructed in the same way as $\sfb(\mu)$ is defined: start from the empty vector, and for $k$ varying from $0$ to $b$, add all the values $k$ from $\sfb$ as rightmost as possible before $f_k$ inclusive. The resulting vector is equal to $\sfb$, for otherwise property~(\ref{def:bracketvectors3}) would be invalidated at some point in the process.   

\item \emph{If $T=\rflush(\mu)$ then $\sfb(T)=\sfb(\mu)$}. 
As we have seen in the proof of~\ref{it:uniquebvector}, a $\nu$-bracket vector is completely determined by the number of values $k$ it contains for each $0\leq k \leq b$. Since $T=\rflush(\mu)$ and $\mu$ have the same number of labels equal to $k$ for each $k$, then their bracket vectors must be equal.  \qedhere
\end{enumerate}
\end{proof}

\subsection{Properties of bracket vectors and proof of Theorem~\ref{thm:bracketlattice}}\label{sec:proofbracketlattice}
Let $T$ and~$T'$ be two $\nu$-trees. We will write $T\rightarrow T'$ if the tree $T'$ can be obtained from $T$ by a sequence of right rotations. 
The rotation action on bracket vectors can be described as follows.

\begin{lemma}\label{lem:bracketvectorrotation}
Let $T'$ be a $\nu$-tree obtained by a right rotation of $T$ at a node with label $x$. The bracket vector $\sfb(T')$ can be obtained from $\sfb(T)$ by replacing the first appearance of $x$ by the value $y$ at position $f_x+1$ (the value following the last $x$).
\end{lemma}

\begin{proof}
This result follows from our schematic illustration of right rotation in Figure~\ref{fig:bracket_rotation}: The first value $x$ in $\sfb(T)$ is the node being rotated, while the last $x$ corresponds to the last node of the subtree $B$ read in in-order. This last $x$ stays at the fixed position $f_k$ for all trees. The first entry $y$ after the last $x$ corresponds to the label of the parent of the node being rotated, which gets a new label equal to~$y$.  
\end{proof}

\begin{corollary}\label{lem:bracketvectors_properties1}
If $T\rightarrow T'$ then $\sfb(T)< \sfb(T')$.
\end{corollary}
\begin{proof}
Since $y>x$ in the previous lemma, applying a right rotation to a tree acts on its bracket vector by increasing exactly one of its entries. The result follows by applying a sequence of rotations. 
\end{proof}

Note that if a $\nu$-bracket vector $\sfb$ has at least two $x$'s and $x<n$, then a right rotation action can always be performed at the first appearance of $x$, replacing it by the value $y$ that appears after the last $x$ in $\sfb$.

\begin{lemma}\label{lem:bracketvectors_properties2}
If $\sfb(T)< \sfb(T')$ then $T\rightarrow T'$.
\end{lemma}

\begin{proof}
Let $\sfb(T)=(b_1,\dots,b_{\ell+1})$ and $\sfb(T')=(b_1',\dots,b_{\ell+1}')$ and consider the smallest index $i$ such that $b_i\neq b_i'$. For simplicity we call $b_i=x$ and $b_{f_x+1}=y$ the first entry after the last $x$ in $\sfb$. We also denote $b_i'=z$ and $b_{f_x+1}'=w$. 

\begin{table}[htbp]
\begin{center}
\begin{tabular}{ccccccccc}
$\sfb' =$&\dots & $z$ & \dots & $\circled{x}$ & $w$ & \dots & $\circled{z}$ & \dots \\ 
$\sfb =$&\dots & $x$ & \dots & $\circled{x}$ & $y$ & \dots & $\circled{z}$ & \dots \\ 
&&$i$&&$f_x$&&&$f_z$&
\end{tabular}
\end{center}
\end{table}%

We start by observing that $b_i$ is the first entry equal to $x$ in $\sfb$, otherwise there would be a pattern $x\dots z \dots x$ with $x<z$ in $\sfb'$, which would contradict Property~(\ref{def:bracketvectors3}) in the definition of $\nu$-bracket vectors. 
Applying a rotation operation on this first $x$ produces a new bracket vector $\bar \sfb$, which is obtained from $\sfb$ by replacing its first $x$ by the value $y$.
Since $y\leq w$ (because $\sfb<\sfb'$) and $w\leq z$ by Property~(\ref{def:bracketvectors3}) for $\sfb'$, 
we get that $y \leq z$. Therefore $\sfb<\bar \sfb \leq \sfb'$. If $\bar \sfb = \sfb'$ we are done, otherwise we continue doing rotations until reaching $\sfb'$ in a finite number of steps. 
\end{proof}

We are now ready to prove Theorem~\ref{thm:bracketlattice} asserting that the $\nu$-Tamari lattice is isomorphic to the lattice of $\nu$-bracket vectors under componentwise order.

\begin{proof}[Proof of Theorem~\ref{thm:bracketlattice}]
By Theorem~\ref{thm:rotationlatticetrees}, the $\nu$-Tamari lattice is isomorphic to the lattice of $\nu$-trees whose order is induced by right rotations. This lattice is isomorphic to the lattice of $\nu$-bracket vectors by Proposition~\ref{prop:bracketvectors}, Corollary~\ref{lem:bracketvectors_properties1} and Lemma~\ref{lem:bracketvectors_properties2}. 
\end{proof}

\subsection{Meet and join}\label{sec:meetjoin}

The properties~(\ref{def:bracketvectors1}),~(\ref{def:bracketvectors2}), and~(\ref{def:bracketvectors3}) in the definition of $\nu$-bracket vectors are clearly preserved after taking the componentwise minimum between two bracket vectors. This gives us a simple description for the meet operation:

\begin{proposition}\label{prop:meet}
The meet of two $\nu$-bracket vectors $\sfb = (b_1 , \dots, b_{\ell+1} )$ and
$\sfb' = (b'_1 , \dots , b'_{\ell+1} )$ is their component-wise minimum 
 \[\sfb\wedge \sfb'= (\min\{b_1,b_1'\},\dots , \min\{b_{\ell+1},b_{\ell+1}'\}).\] 
\end{proposition}

The join cannot be obtained by taking the componentwise maximum. Instead, it can computed in terms of the meet of the corresponding reflected trees $\reverse T$ and~$\reverse T'$. 

\begin{proposition}
The join $T\vee T'=\reverse{\reverse T \wedge \reverse T'}$. 
\end{proposition}

\begin{proof}
As in the proof of Corolary~\ref{cor:combinatorialduality}, the map $T \rightarrow \reverse T$ sends the $\nu$-Tamari lattice $\Tam{\nu}$ to the dual of $\Tam{\reverse \nu}$. The result follows.
\end{proof}

An example of the meet and join operation using bracket vectors is illustrated in Figure~\ref{fig:meetjoin}.

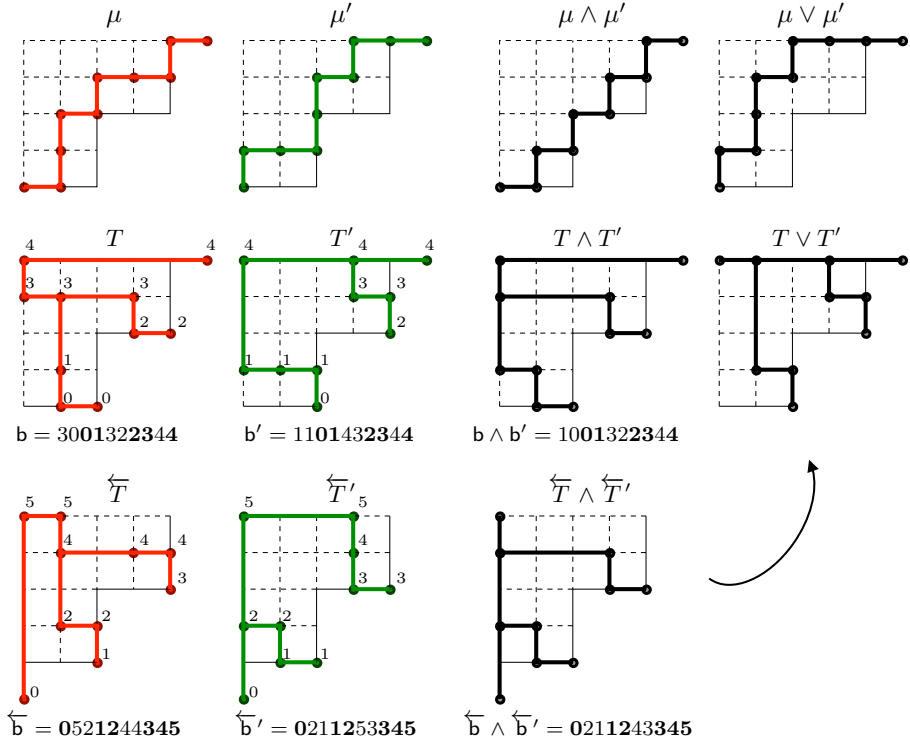
\begin{figure}[htbp]
	\begin{overpic}[width=\textwidth]{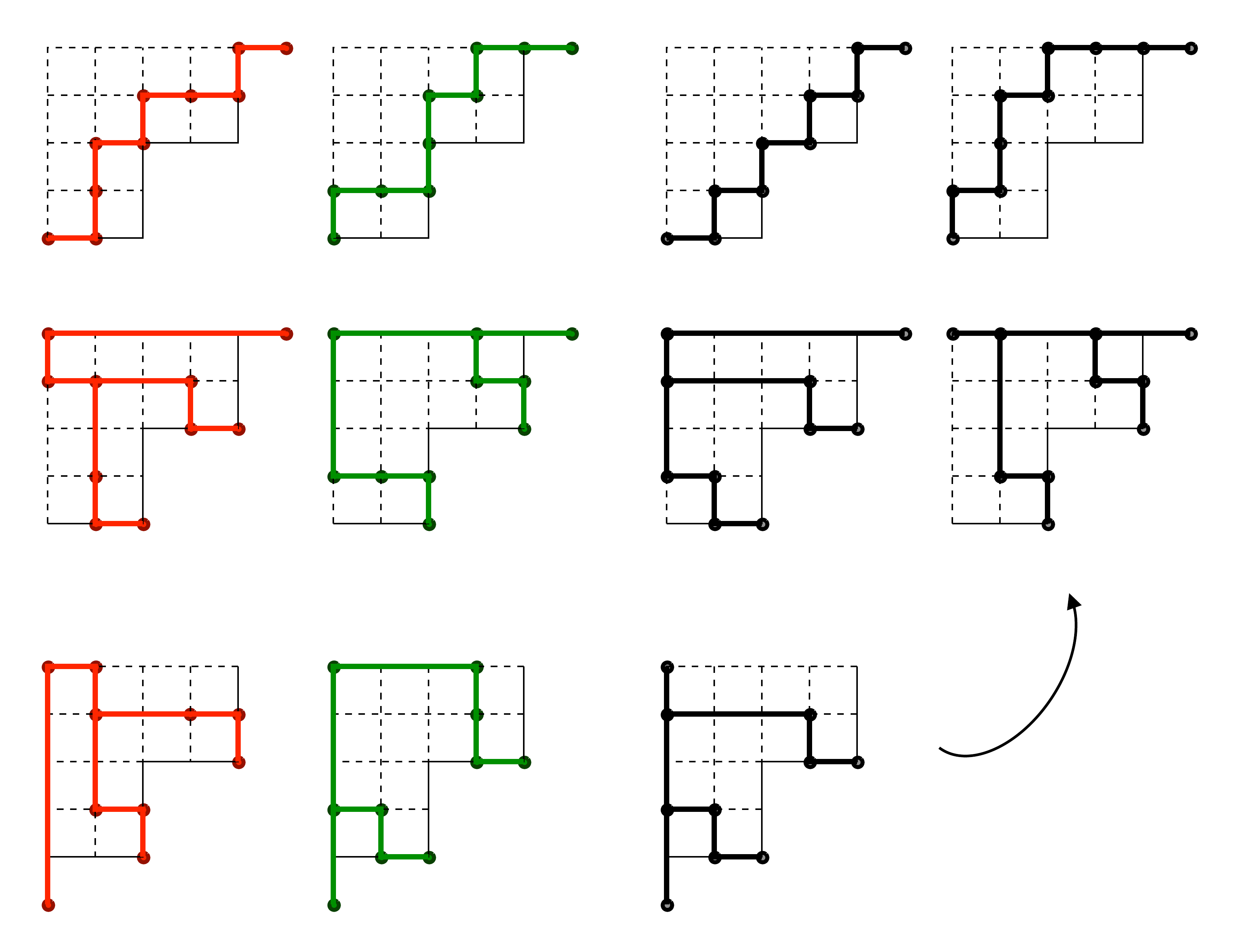}
	\put(12.5,75){$\mu$}
	\put(36,75){$\mu'$}
	\put(60,75){$\mu\wedge\mu'$}
	\put(83,75){$\mu\vee\mu'$}
	\put(12.5,51.5){\small $T$}
	\put(36,51.5){\small $T'$}
	\put(59.5,51.5){\small $T\wedge T'$}
	\put(82.5,51.5){\small $T\vee T'$}
	\put(3,31){\footnotesize $\sfb = 30{\bf0}{\bf1}32{\bf2}{\bf3}4{\bf4}$}
	\put(27,31){\footnotesize $\sfb'=11{\bf0}{\bf1}43{\bf2}{\bf3}4{\bf4}$}
	\put(51,31){\footnotesize $\sfb \wedge \sfb' = 10{\bf0}{\bf1}32{\bf2}{\bf3}4{\bf4}$}
	\put(12.5,24.5){\small $\reverse T$}
	\put(35.5,24.5){\small $\reverse T'$}
	\put(59,24.5){\small $\reverse T\wedge \reverse T'$}
	\put(2,0){\footnotesize $\reverse\sfb = {\bf0}52{\bf1}{\bf2}44{\bf3}{\bf4}{\bf5}$}
	\put(26,0){\footnotesize $\reverse\sfb'={\bf0}21{\bf1}{\bf2}53{\bf3}{\bf4}{\bf5}$}
	\put(50,0){\footnotesize $\reverse\sfb \wedge \reverse\sfb' = {\bf0}21{\bf1}{\bf2}43{\bf3}{\bf4}{\bf5}$}
	\put(12,35){\tiny$0$}
	\put(8,35){\tiny$0$}
	\put(8,39){\tiny$1$}
	\put(8,47){\tiny$3$}
	\put(4,47){\tiny$3$}
	\put(16,47){\tiny$3$}
	\put(16,43){\tiny$2$}
	\put(20,43){\tiny$2$}
	\put(4,51){\tiny$4$}
	\put(23,51){\tiny$4$}
	\put(27,51){\tiny$4$}
	\put(39,51){\tiny$4$}
	\put(46,51){\tiny$4$}
	\put(39,47){\tiny$3$}
	\put(43,47){\tiny$3$}
	\put(43,43){\tiny$2$}
	\put(27,39){\tiny$1$}
	\put(31,39){\tiny$1$}
	\put(35,39){\tiny$1$}
	\put(35,35){\tiny$0$}
	\put(4,24){\tiny$5$}
	\put(8,24){\tiny$5$}
	\put(8,20){\tiny$4$}
	\put(16,20){\tiny$4$}
	\put(20,20){\tiny$4$}
	\put(20,16){\tiny$3$}
	\put(8,12){\tiny$2$}
	\put(12,12){\tiny$2$}
	\put(12,8){\tiny$1$}
	\put(4.5,4){\tiny$0$}
	\put(27,24){\tiny$5$}
	\put(39,24){\tiny$5$}
	\put(39,20){\tiny$4$}
	\put(39,16){\tiny$3$}
	\put(43,16){\tiny$3$}
	\put(27.5,12){\tiny$2$}
	\put(31,12){\tiny$2$}
	\put(31,8){\tiny$1$}
	\put(35,8){\tiny$1$}
	\put(27.5,4){\tiny$0$}
	\end{overpic}
\caption{Computing the meet and join with bracket vectors.}
\label{fig:meetjoin}
\end{figure}

\section{The \texorpdfstring{$\nu$}{v}-Tamari lattice via subword complexes}
\label{sec:subwordComp}

The main goal of this section is to show that the $\nu$-Tamari lattice is isomorphic to the increasing-flip poset of a suitably chosen subword complex (Theorem~\ref{thm:TamariSubwordComplex}). 
This will be achieved through an identification of $\nu$-trees with certain reduced pipe dreams, which are closely related to the work of Rubey~\cite{rubey_maximal_2012} and Serrano and Stump~\cite{serrano_maximal_2012}. 

A \defn{pipe dream} is defined as a filling of a triangular shape with crosses \cross{} and elbows~\elbow{} so that all pipes (or lines) entering on the left side exit on the top side, see Figure~\ref{fig:vtree_pipedream} (right). 
Given a pipe dream $P$, we label the left ends of the lines with the numbers $1,2,\ldots$ from top to bottom, and transport these labels along the lines to get a labeling of top ends. We denote by $\pi(P)$ the permutation whose one-line representation is given by the top labels, read from left to right. 
In our example from Figure~\ref{fig:vtree_pipedream} (right), the permutation is $\pi(P)=[1,4,3,5,2,6]$.

A pipe dream is called \defn{reduced} if any two pipes have at most one intersection. Reduced pipe dreams play a fundamental role in the combinatorial understanding of Schubert polynomials; they were first considered by Fomin and Kirillov in~\cite{FominKirillov93} as a ``planar history'' of the inversions in a permutation, introduced as rc-graphs by Bergeron and Billey in~\cite{bergeron_rc-graphs_1993}, and further studied using the pipe dream terminology by Knutson and Miller in~\cite{knutson_grobner_2005}. Each crossing \cross{} is meant to represent the action of a transposition of the symmetric group, and the product of the transpositions associated to the crossings~\cross{} (in suitable order) gives a \defn{reduced expression} for~$\pi(P)$ (cf. Section~\ref{sec:pipeequalsubword}).

For the purpose of this section, we view the set of lattice points $\A_\nu$ weakly above a lattice path $\nu$ as lattice points of the smallest square grid such that all points in~$\A_\nu$ are strictly above the main diagonal of the grid, as in Figure~\ref{fig:vtree_pipedream} (left).

With this convention in mind, given a $\nu$-tree $T$, replace each point in $\A_\nu$ by an elbow \elbow{} if it belongs to $T$, and by a crossing \cross{} otherwise. Further replace each point above the main diagonal that is strictly below $\nu$ by an elbow \elbow{}. We obtain a pipe dream fitting inside a triangular shape $(n,n-1,\ldots,2,1)$. We denote by $\pi_\nu(T)$ its corresponding permutation.
An example is illustrated in Figure~\ref{fig:vtree_pipedream} (center).

\begin{figure}[htbp]
\includegraphics[width=\textwidth]{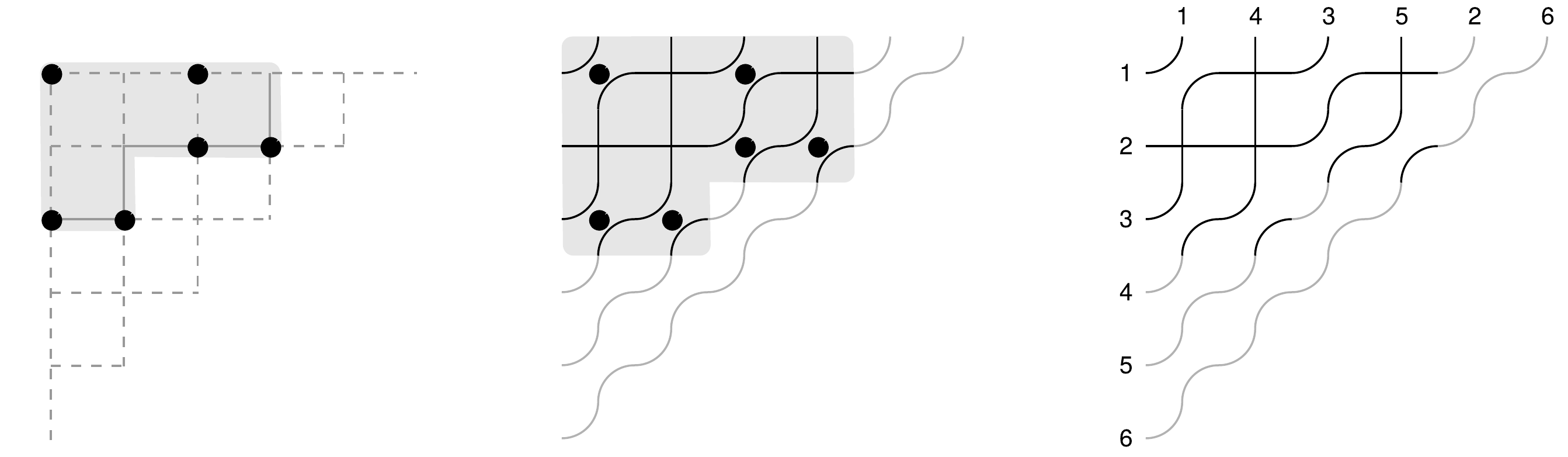}
\caption{The pipe dream representation of a $\nu$-tree.}
\label{fig:vtree_pipedream}
\end{figure}

\begin{proposition}[{\cite{bergeron_rc-graphs_1993}}]
\label{prop:vtreesfacets}
For a fixed $\nu$, the permutation $\pi_\nu:=\pi_\nu(T)$ is independent of the $\nu$-tree $T$. Moreover, 
$\nu$-trees give all reduced pipe dreams for~${\pi_\nu}$. 
\end{proposition}

\begin{proof}
In the language of reduced pipe dreams, right rotations on $\nu$-trees correspond to \defn{(general) chute moves} (as defined and illustrated in Figure~\ref{fig:chutemove})\footnote{These chute moves are slightly more general than the (two sided) chute moves originally defined by Bergeron and Billey in~\cite{bergeron_rc-graphs_1993}. This more general version was defined by Rubey in~\cite{rubey_maximal_2012}.}. 
\begin{figure}[htbp]
	\begin{overpic}[width=0.8\textwidth]{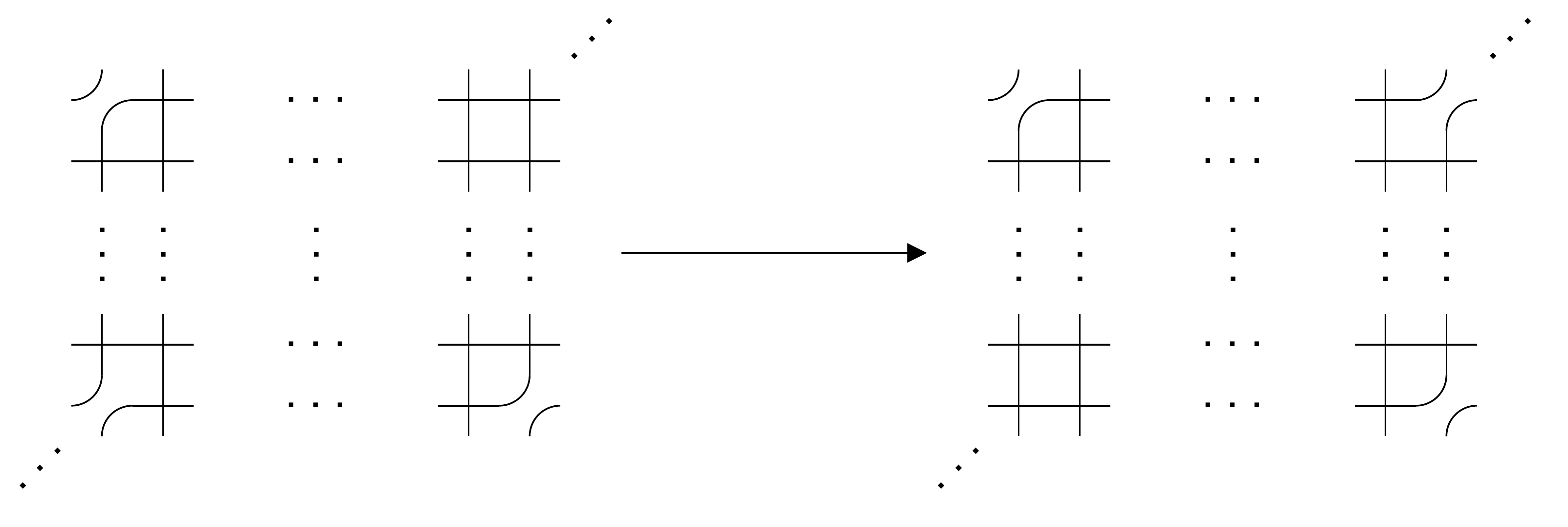}
	\put(40,19){Chute move}
	\end{overpic}
\caption{Rotations of $\nu$-trees correspond to chute moves in pipe dreams.}
\label{fig:chutemove}
\end{figure}
Since such moves do not alter the permutation of the pipe dream~\cite[Lemma 3.5]{bergeron_rc-graphs_1993}, different $\nu$-trees give rise to the same permutation. 
Since all reduced pipe dreams of a permutation are connected by chute and inverse chute moves~\cite[Thm.~3.7]{bergeron_rc-graphs_1993}, $\nu$-trees and reduced pipe dreams for $\pi_\nu$ coincide.
In order to show that the pipe dreams are reduced, it suffices to check it for one tree; this holds for $\Tmin$.
\end{proof}

\begin{remark}
Proposition~\ref{prop:vtreesfacets} is a special case of Rubey's result~\cite[Theorem~3.2]{rubey_maximal_2012} where $k=1$ and the polyomino is chosen to be a Ferrers shape, as well as the special case of Serrano and Stump's result~\cite[Theorem 2.1]{serrano_maximal_2012} for $k=1$.  
\end{remark}

\subsection{Reduced pipe dreams as facets of subword complexes}
\label{sec:pipeequalsubword}

Reduced pipe dreams for a permutation $w$ can be identified with the facets of certain subword complex~\cite[Section 1.8]{knutson_grobner_2005}. These complexes were introduced by Knutson and Miller for Coxeter groups in~\cite{knutson_subword_2004}, and reduced pipe dreams are a special case corresponding to the symmetric group. 

Let us briefly recall some basic notions relating to subword complexes. Since we are only working with the symmetric group, we restrict our presentation to this level of generality. Let $\frS_{n+1}$ be the symmetric group of permutations of $[n+1]$, and $S=\{s_1,\dots,s_n\}$ be the generating set of simple transpositions $s_i= (i\ i+1)$. Every element~$w\in \frS_{n+1}$ can be written as a product $w=s_{i_1}s_{i_2}\dots s_{i_k}$ of elements in $S$. If~$k$ is minimal among all such expressions for $w$, then $k$ is called the \defn{length} $\ell(w)$ of $w$, and $s_{i_1}s_{i_2}\dots s_{i_k}$ is called a \defn{reduced expression} for $w$. 

\begin{definition}[\cite{knutson_subword_2004}]\label{defn:subword}
Let~$Q=(q_1,\dots,q_m)$ be a word in $S$ and $\pi\in \frS_{n+1}$ be an element of the group. The \defn{subword complex}~$\subwordComplex{Q}{\pi}$ is a simplicial complex whose facets (maximal faces) are given by subsets $I\subset [m]=\{1,2,\ldots,m\}$, such that the subword of~$Q$ with positions at~$[m]\ssm I$ is a reduced expression of~$\pi$.

Two facets $I$ and $J$ are \defn{adjacent} if they differ by one single element, that is $I\smallsetminus i=J\smallsetminus j$. The operation of replacing $i$ by $j$ to go from $I$ to $J$ is called a \defn{flip}.   
The flip from $I$ to $J$ is called \defn{increasing} if $i<j$.
The \defn{increasing flip poset} of~$\subwordComplex{Q}{\pi}$ is the partial order on its facets, whose covering relations correspond to increasing flips. The \defn{facet adjacency graph} of~$\subwordComplex{Q}{\pi}$ is the graph whose vertices are facets of~$\subwordComplex{Q}{\pi}$ and edges correspond to pairs of adjacent facets. 
\end{definition}

\begin{example}
Let $n=2$ and $S=\{s_1,s_2\}=\{(1\ 2), (2\ 3)\}$. Let $\pi=[2,3,1]=s_1s_2$ and $Q=(q_1,q_2,q_3,q_4,q_5)=(s_1,s_2,s_1,s_2,s_1)$. 
Since the reduced expressions of $\pi$ in $Q$ are given by $q_1q_2=q_1q_4=q_3q_4=\pi$, 
the facets of~$\subwordComplex{Q}{\pi}$ are $\{3,4,5\},\{2,3,5\}$ and $\{1,2,5\}$. 
The increasing flips are:
\[
\{1,2,5\} \rightarrow \{2,3,5\} \rightarrow \{3,4,5\}. 
\]
This subword complex is illustrated in Figure~\ref{fig:A2}.

\begin{figure}[htbp]
\begin{center}
\includegraphics[width=0.4\textwidth]{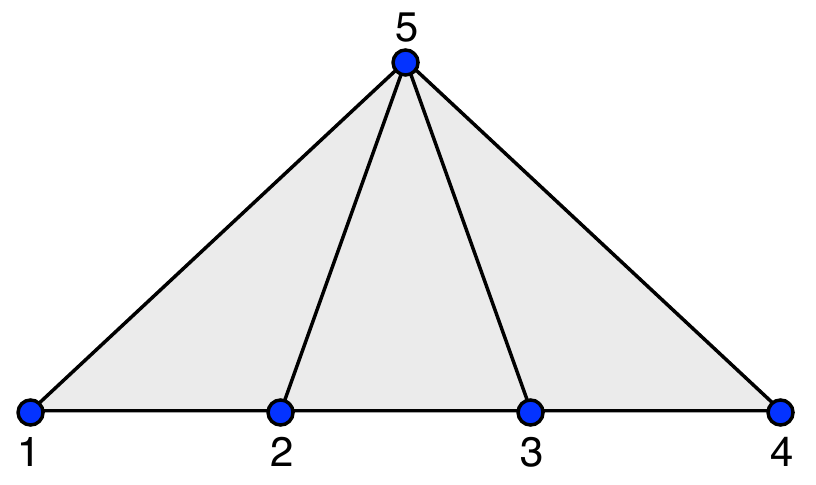}
\caption{Subword complex~$\subwordComplex{Q}{\pi}$ for $Q=(s_1,s_2,s_1,s_2,s_1)$ and $\pi=s_1s_2$. Its maximal faces are $\{3,4,5\},\{2,3,5\}$ and $\{1,2,5\}$.}
\label{fig:A2}
\end{center}
\end{figure}

\end{example}

For a fixed lattice path $\nu$, recall that~$\F_\nu$ is the Ferrers diagram that lies weakly above $\nu$. For a lattice point $p$ in $\F_\nu$, denote by $d(p)$ the lattice distance from $p$ to the top-left corner of $\F_\nu$. Set $\hat d=\max_{p\in\F_\nu}d(p)$. We denote by $\pi_\nu$ the permutation in $\frS_{\hat d+2}$ whose Rothe diagram (i.e.\ the set $\{(\pi(j),i)\colon i<j, \pi(i)>\pi(j)\}$) is equal to $\F_\nu$ with its northwest block lying at (2,2).
Now label each integer lattice point $p$ in $\F_\nu$ by the transposition~$s_{d(p)+1}$. Define $Q_\nu$ as the word obtained 
by reading the labels of each row from left to right, and the rows from bottom to top.
 See Figure~\ref{fig:vTamariSubwordcomplexes} (compare~\cite{serrano_maximal_2012}).
 
 \begin{figure}[bht]
\begin{center}
\includegraphics[width= 0.85\textwidth]{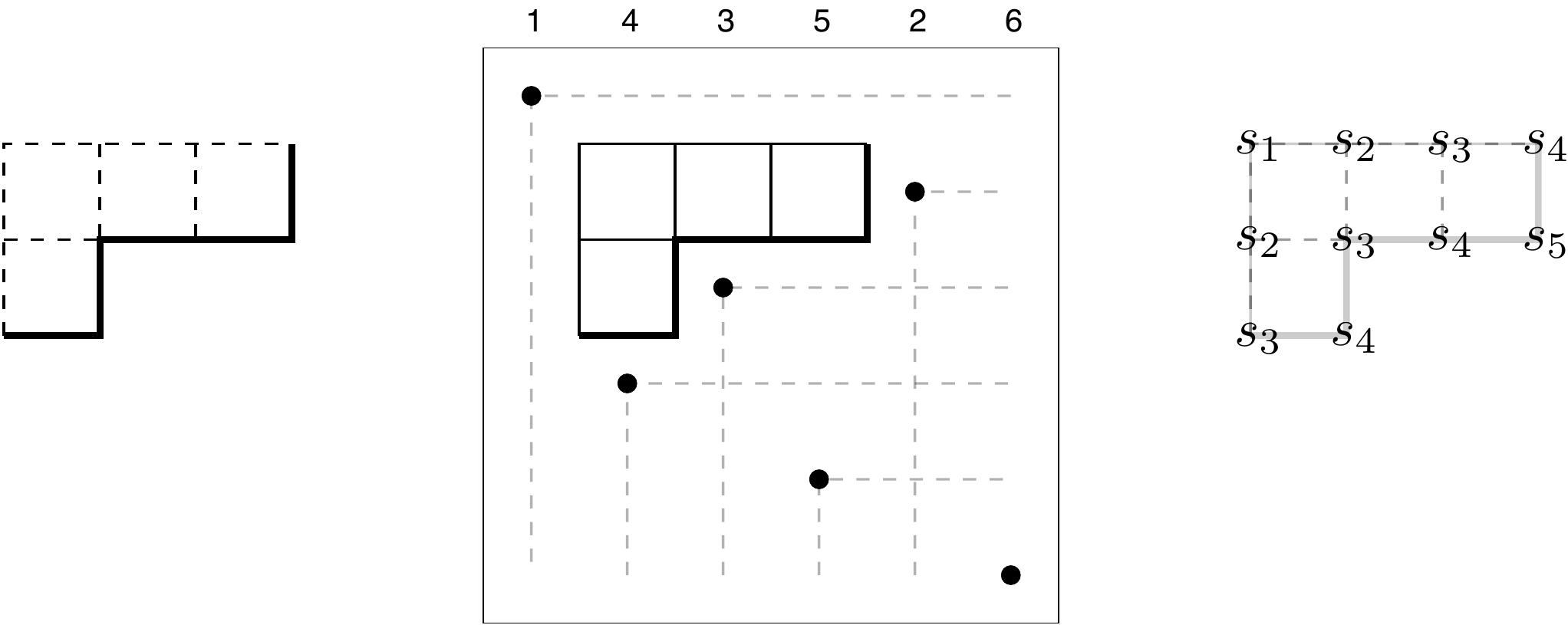}
\caption{Lattice path $\nu=\sfE\sfN\sfE\sfE\sfN$ and its corresponding Ferrers diagram $\F_\nu$(left). Rothe diagram of the permutation $\pi_\nu=[1,4,3,5,2,6]$ (middle). Corresponding word 
$Q_\nu=(s_3,s_4, s_2, s_3, s_4, s_5, s_1, s_2, s_3, s_4)$
(right).}\label{fig:vTamariSubwordcomplexes}
\end{center}
\end{figure}

Thus, from a $\nu$-tree $T$ one gets a reduced expression for $\pi_\nu$ as the product of the transpositions in $Q_\nu$ corresponding to points of $\A_\nu$ \textbf{not} in $T$. Figure~\ref{fig:vtree_expression} illustrates this, along with the effect of a rotation.

\begin{figure}[htbp]
\begin{center}
	\begin{overpic}[width=0.6\textwidth]{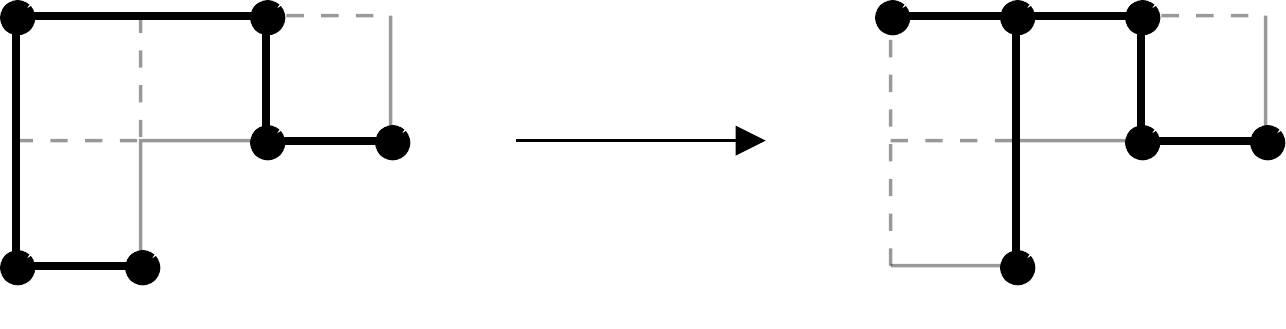}
		\put(-8,-1){$s_2s_3s_2s_4=$ {\footnotesize $[1,4,3,5,2,6]$}}
		\put(1.5,11.5){$s_2$}
		\put(11,11.5){$s_3$}
		\put(11,21.5){$s_2$}
		\put(30,21.5){$s_4$}
		\put(60,-1){$s_3s_2s_3s_4=$ {\footnotesize $[1,4,3,5,2,6]$}}
		\put(67.5,13.5){$s_2$}
		\put(79.5,13.5){$s_3$}
		\put(67.5,4.5){$s_3$}
		\put(96,21.5){$s_4$}		
		\put(41.5,18){\small rotation}
	\end{overpic}
\caption{Complements of $\nu$-trees are the reduced expressions of $\pi_\nu$ in $Q_\nu$.}
\label{fig:vtree_expression}
\end{center}
\end{figure}
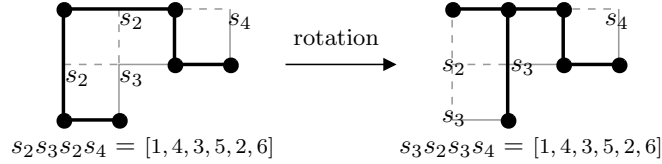

\begin{theorem}\label{thm:TamariSubwordComplex}
The $\nu$-Tamari lattice is isomorphic to the increasing flip poset of the subword complex $\subwordComplex{Q_\nu}{\pi_\nu}$.
\end{theorem}

\begin{proof}
By Theorem~\ref{thm:rotationlatticetrees}, the $\nu$-Tamari lattice is isomorphic to the rotation lattice of $\nu$-trees. The facets of the subword complex $\subwordComplex{Q_\nu}{\pi_\nu}$ are in correspondence with $\nu$-trees by Proposition~\ref{prop:vtreesfacets}. Two facets are related by an increasing flip if and only if the corresponding $\nu$-trees are related by a right rotation by~Lemma~\ref{lem:flipsequalrotations}.
\end{proof}

\begin{remark}
Theorem~\ref{thm:TamariSubwordComplex} is equivalent to Theorem~\ref{thm:rotationlatticetrees} and Lemma~\ref{lem:flipsequalrotations} together with any of the following two results:
\begin{itemize}
\item The specialization of Rubey's result~\cite[Theorem~3.2]{rubey_maximal_2012} for $k=1$ and the polyomino being a Ferrers shape.
\item The specialization of Serrano and Stump's result~\cite[Theorem 2.1]{serrano_maximal_2012} for $k=1$.  
\end{itemize}
Note that Lemma~\ref{lem:flipsequalrotations} implies that all flips between pipe dreams with permutation~$\pi_\nu$ are (general) chute moves. 
This is not true for arbitrary permutations, which exhibit flips that do not correspond to chute moves.
\end{remark}

\begin{remark}\label{rem_vTamaricomplex}
The set of $\nu$-trees is naturally equipped with the simplicial complex structure of the corresponding subword complex, where covering pairs of $\nu$-trees represent adjacent facets. This generalization of the simplicial associahedron was already considered in the context of the $\nu$-Tamari lattice in~\cite{tams2017}.
\end{remark}

\subsection{Rubey's lattice conjecture}
The collection of reduced pipe dreams of a permutation $w$ can be equipped with a natural poset structure determined by (general) chute moves. Rubey formulated the following conjecture in~\cite[Conjecture~2.8]{rubey_maximal_2012}.

\begin{conjecture}[Rubey's Lattice Conjecture~\cite{rubey_maximal_2012}]
The poset of reduced pipe dreams of a permutation $w$, whose covering relations are defined by (general) chute moves, is a lattice.
\end{conjecture}

An important class of permutations arising from the theory of Schubert polynomials is the class of {dominant permutations}, see for instance~\cite{manivel_symmetric_2001}.
A \defn{dominant permutation} is a permutation whose Rothe diagram is the shape of a partition with its northwest block located at position $(1,1)$. Those are permutations avoiding the pattern $132$. 
For $u\in \mathfrak{S}_m$ and $v\in \mathfrak{S}_n$ we denote by $u\oplus v\in \mathfrak{S}_{m+n}$
the permutation defined by 
\[
(u\oplus v)(i) =
  \begin{cases}
    u(i), & \text{for } 1 \leq i \leq m \\
    v(i)+m, & \text{for } m < i \leq m+n.
  \end{cases}
\]

The collection of permutations $\pi_\nu$ associated to lattice paths $\nu$ are exactly the permutations of the form $w=1\oplus u$, where $u$ is a dominant permutation. As a direct consequence of Theorem~\ref{thm:TamariSubwordComplex} we get that Rubey's conjecture holds for a special class of permutations determined by dominant permutations:

\begin{theorem}
\label{thm_RubeysConjecture}
Rubey's Lattice Conjecture holds for permutations $w=1\oplus u$ where~$u$ is a dominant permutation.
\end{theorem}
\begin{proof}
By Theorem~\ref{thm:TamariSubwordComplex} and Lemma~\ref{lem:flipsequalrotations}, the poset of reduced pipe dreams of $w$ is isomorphic to a $\nu$-Tamari lattice, which is known to be lattice.
\end{proof}

\subsection{The Edelman--Greene correspondence}
Using Edelman--Greene insertion, Woo~\cite{woo_catalan_2004} described a bijection between pipe dreams with permutation $[1,n+1,n,\dots,2]$ and Dyck paths with $2n$ steps. His bijection extends trivially to a bijection between pipe dreams with permutation $w=1\oplus u$ and $\nu$-paths, where $u$ is a dominant permutation and $\nu$ is the path whose Ferrers diagram $\F_\nu$ equals the Rothe diagram of $u$. 
As noticed in Proposition~\ref{prop:vtreesfacets}, pipe dreams with permutation $w=1\oplus u$ are in correspondence with $\nu$-trees. The purpose of this section is to show that Woo's bijection from $\nu$-trees (when regarded as pipe dreams) to $\nu$-paths coincides with the left flushing bijection $\lflush$ from Section~\ref{sec:RotationlatticeTamarilatticeIsomorphism}. 

\begin{remark}
In~\cite{serrano_maximal_2012}, Serrano and Stump extended Woo's result to a bijection between pipe dreams with permutation $w=[1,2,\dots,k-1] \oplus u$ and $k$-tuples of nested $\nu$-paths, and used it to describe a bijection between $k$-triangulations of a polygon and $k$-tuples of nested Dyck paths. 
\end{remark}

Let $\nu$ be path and $w=1\oplus u$, where $u$ is the dominant permutation whose Rothe diagram is equal to $\F_\nu$. Edelman-Greene's (column) correspondence associates to each $\nu$-tree $T$ a pair $(\X,\Y)$ of an insertion tableau and a recording tableau, as follows:

Let $T$ be a $\nu$-tree. We denote by~$\A_\nu$ be the set of lattice points weakly above~$\nu$ in~$\F_\nu$. 
Each point in $\A_\nu$ has a coordinate $(i,j)$ where $i$ stands for the $i$th row from top to bottom and $j$ stands for the $j$th column from left to right. 
The \defn{reading biword} of $T$ is the array ${a_1,\dots, a_\ell}\choose{b_1,\dots,b_\ell}$ obtained by reading $i \choose i+j-1$ for every point $(i,j)$ in the complement $T^c:=\A_\nu\smallsetminus T$, row by row from right to left and from top to bottom.  

We insert the letters of the word formed by the bottom row using column Edelman--Greene insertion~\cite{edelmangreene_balanced_1987} into a tableau, while recording the corresponding letters from the first row. This produces an insertion tableau $\X=\X(T)$ and a recording tableau $\Y=\Y(T)$. An example is illustrated in Figure~\ref{fig_EdelmanGreene}.

We briefly recall the column \defn{Edelman--Greene insertion} for completeness. When we insert a letter $i$ into a column:
\begin{itemize}
\item if all numbers in that column are smaller than or equal to~$i$, we append $i$ to that column;
\item if the column contains both $i$ and $i+1$, it remains unchanged and an $i+1$ is bumped to the next column;
\item otherwise, we replace the smallest number $j$ greater than $i$ in that column by $i$ and bump $j$ to the next column. 
\end{itemize}

\begin{figure}[htbp]
\begin{center}
\includegraphics[width= 0.8\textwidth]{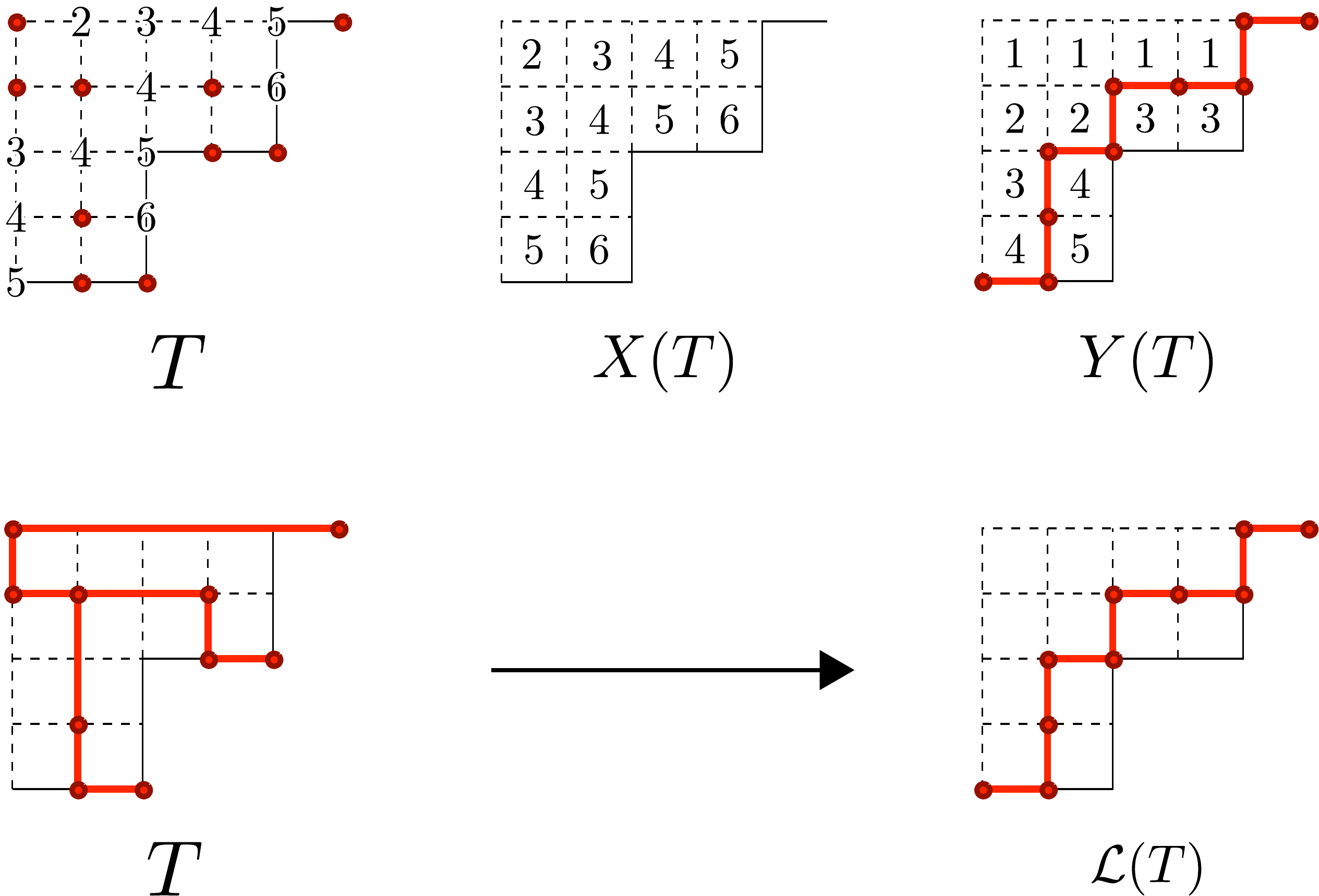}
\caption{The $\nu$-path obtained applying the column Edelman-Greene correspondence to a $\nu$-tree $T$ coincides with the left flushing $\lflush (T)$. The reading biword in this case is~\usebox{\smlmat}.}
\label{fig_EdelmanGreene}
\end{center}
\end{figure}

\begin{lemma}[{\cite[Lemma~3.1]{serrano_maximal_2012}}]
For every $\nu$-tree $T$, the shapes of $\X(T)$ and $\Y(T)$ are given by $\F_\nu$.
\end{lemma}

\begin{lemma}[{\cite[Proposition~3]{woo_catalan_2004}}]\label{lem:kthrowEG}
The $k$th row of the tableau $\Y(T)$ contains only entries $k$ or $k+1$. 
\end{lemma}

\begin{definition}[{\cite{woo_catalan_2004}}]
For a $\nu$-tree $T$, define $\EG(T)$ to be the $\nu$-path such that the boxes weakly above it are precisely those whose row number matches their label in $\Y(T)$. Figure~\ref{fig_EdelmanGreene} illustrates an example.
\end{definition}

\begin{proposition}\label{prop:EdelmanGreene}
For every $\nu$-tree $T$, we have $\EG(T)=\lflush(T)$.
\end{proposition}

\begin{proof}
We denote by $T_k$ (resp. $T_k^c$) the points in $T$ (resp. $T^c$) that are in row $k$. 
For a $\nu$-path $\mu$ we denote by $\lambda_k(\mu)$ the number of boxes above $\mu$ in row $k$ (from top to bottom). By Lemma~\ref{lem:kthrowEG}, the values $1,2,\dots, k $ fill all the boxes in the tableau $\Y(T)$ that are above $\nu$ in the first $k-1$ rows and part of boxes in the $k$th row. Therefore, we have:
\begin{equation}\label{eq_edelmangreen_partition}
\lambda_k(\EG(T)) = 
\sum_{i=1}^{k} |T_k^c| -
\sum_{i=1}^{k-1} \lambda_i(\nu).
\end{equation}
On the other hand, $\lambda_k(\lflush(T))$ counts the number of columns that are forbidden by rows below row $k$ in the left flushing bijection, and so:
\begin{equation}\label{eq_flushing_partition}
\lambda_k(\lflush(T)) = 
\sum_{i>k} (|T_i|-1).
\end{equation}
We need to show that $\lambda_k(\EG(T)) = \lambda_k(\lflush(T))$.
Note that the difference on the right hand side of Equation~\eqref{eq_edelmangreen_partition} is independent of the position of the vertices in $T$ located in the first $k$ rows. 
Moreover, if there were no forbidden columns produced by rows below row $k$, then this difference would be equal to zero. 
Now, each forbidden column produced by a row below row $k$ increases this difference by one. Therefore,  
\begin{equation*}
\sum_{i=1}^{k} |T_k^c| - \sum_{i=1}^{k-1} \lambda_i(\nu) =
\sum_{i>k} (|T_i|-1).
\end{equation*}
This finishes the proof.
\end{proof}

\section{Multi \texorpdfstring{$\nu$}{v}-Tamari complexes}\label{sec:geom_multiasso}

For any integer $k\geq 1$ one may define a $(k,\nu)$-tree as a maximal subset of~$\A_\nu$ without $k+1$ pairwise $\nu$-incompatible elements. The \defn{$\knu$-Tamari complex} is the simplicial complex on $\A_\nu$ whose facets are $(k,\nu)$-trees. 
This object was introduced by Jonsson in~\cite{Jonsson2005}.
For $\nu=(NE)^n$ this coincides with the simplicial complex of $(k+1)$-crossing-free subsets of diagonals of a convex $(n+2)$-gon. 
This complex is conjectured to be realizable as the boundary complex of a simplicial polytope~\cite{Jonsson2005}, whose dual would be a simple polytope $\Delta_{n+2,k}^*$ known as the \defn{simple multiassociahedron} (see the introductions of~\cite{BergeronCeballosLabbe2015} and~\cite{Manneville2017}, and the references therein, for the current knowledge on the existence of these polytopes). 

For $k=1$ we have recently shown that the facet adjacency graphs of $(1,\nu)$-Tamari complexes can be realized as the edge graphs of polytopal subdivisions of (simple) associahedra%
\footnote{Whenever $\nu$ does not have two consecutive non-initial north steps or does not have two consecutive non-final east steps, which holds in particular when the lattice paths weakly above $\nu$ are rational Dyck paths, like Fuss-Catalan paths.}%
~\cite{tams2017}.
We believe that a similar statement might be true for general~$k$. The following proposition is a first positive result in this direction:

\begin{proposition}\label{prop:multi}
Let 
$m\geq k$ and $\nu=(NE^m)^{k+1}$. The facet adjacency graph of the Fuss-Catalan $(k,\nu)$-Tamari complex is the edge graph of a polytopal subdivision of the simple multiassociahedron $\Delta_{2k+2,k}^*$ (a $k$-dimensional simplex).
\end{proposition}

\begin{proof}
 We will show that the facet adjacency graph of the $\knu$-Tamari complex is the edge graph of a fine mixed subdivision of an $(m-k+1)$-fold dilated $k$-dimensional simplex~$(m-k+1)\cdot\Delta_k$, obtained from the staircase triangulation of $\productp{m-k}{k}$ via the Cayley trick~\cite{San05} (we refer to Sections~6.2 and 9.2 of~\cite{DeLoeraRambauSantos2010} for a nice introduction to triangulations of products of simplices, mixed subdivisions, and the Cayley trick).

 Our first observation is that the facet adjacency graph of the $(k,\nu)$-Tamari complex coincides with the facet adjacency graph of the $(k,\mu)$-Tamari complex for $\mu=E^mN^k$. Indeed, all the lattice points $p\in \A_\nu$ beyond the $m$th column, as well as the point in the lowest row, belong to every $\knu$-tree. The reason is that such a point $p$ cannot be contained in a $(k+1)$-subset of pairwise $\nu$-incompatible elements. 
By the same token, points of $\A_\mu$ lying in the first $k$ and in the last $k$ southwest-northeast diagonals weakly above $\mu$ also belong to every $(k,\mu)$-tree. We call such nodes belonging to every $(k,\mu)$-tree \defn{irrelevant nodes}. See Figure~\ref{fig:irrelevant}, where irrelevant nodes are drawn as white-filled dots. 
 
  \begin{figure}[htbp]
 \centering
 \includegraphics[width=.45\textwidth]{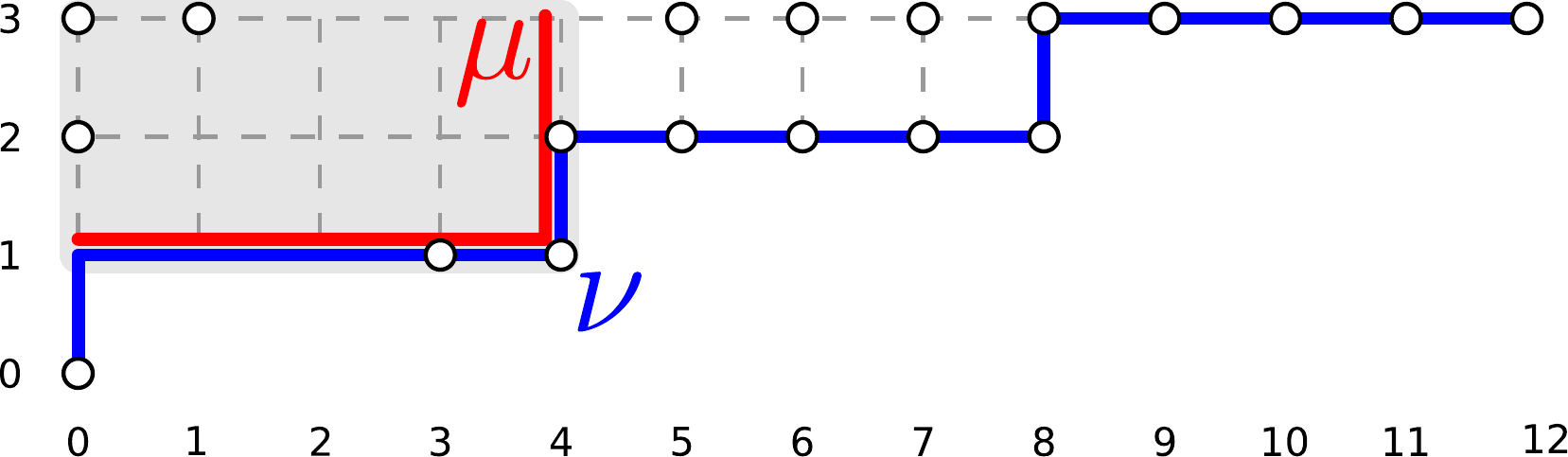} \qquad 
  \includegraphics[width=.45\textwidth]{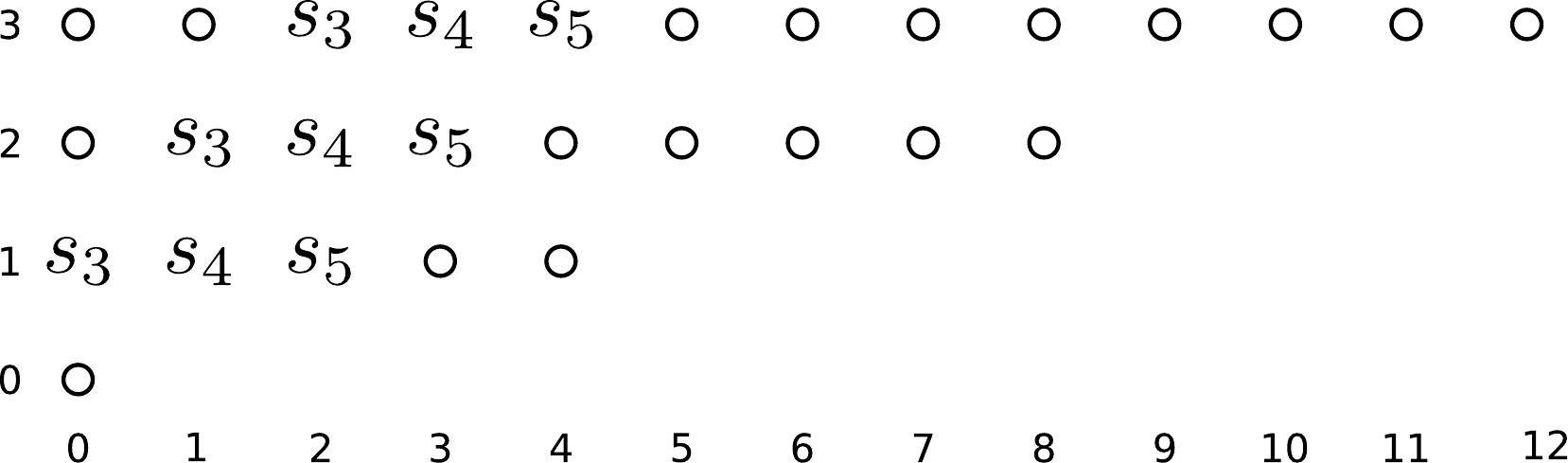}
 \caption{Example for $m=4$ and $k=2$. The white-filled dots in the left figure are irrelevant because they are contained in every facet of the $(k,\nu)$-Tamari complex. The right figure displays the corresponding relevant transpositions for the word $Q_{k,\nu}$.}\label{fig:irrelevant}
 
 \end{figure}

 In~\cite[Theorem~2.1]{serrano_maximal_2012}, Serrano and Stump showed that $\knu$-trees can be viewed as reduced pipe dreams of a certain permutation. In particular, this identifies the $\knu$-Tamari complex as the join of a simplex with a subword complex $\subwordComplex{Q}{\pi}$. The vertices of the simplex correspond to the irrelevant lattice points weakly above~$\nu$, and the vertices of the subword complex to the relevant ones. More precisely, label the lattice points weakly above $\nu$ similarly as in Section~\ref{sec:pipeequalsubword} (see Figure~\ref{fig:irrelevant}), and let $Q=Q_{k,\nu}$ be the word obtained 
 by reading the labels of the \emph{relevant points} of each row from left to right, and the rows from bottom to top.
 \[Q=Q_{k,\nu}=(s_{k+1}, s_{k+2},\dots, s_{m+1})^{k+1}.\]
 The permutation is given by $\pi=\pi_{k,\nu}=s_{k+1} s_{k+2}\cdots s_{m+1}$.
 
 The facet adjacency graph of the $\knu$-Tamari complex is therefore equal to the facet adjacency graph of $\subwordComplex{Q}{\pi}$. Since there is only one possible reduced expression for $\pi$, given by $s_{k+1} s_{k+2}\cdots s_{m+1}$, the reduced expressions of $\pi$ in $Q$ correspond to possible matchings
 \[(s_{k+1},j_1), (s_{k+2},j_2),\dots, (s_{m+1},j_{m+1-k})\]
 such that $1\leq j_1\leq j_2\leq\cdots\leq j_{m+1-k}\leq k+1$, where $j_i$ is the copy of the factor $(s_{k+1}, s_{k+2},\dots, s_{m+1})$ in $Q$ from which $s_{k+i}$ is chosen; or equivalently the height of the lattice point that corresponds to this transposition.
 
 Such matchings can be encoded as subgraphs of the complete bipartite graph $G_{k,\nu}\cong K_{k+1,m+1-k}$ with color classes $S=\{s_{k+1},\ldots,s_{m+1}\}$ and $\{1,\ldots,k+1\}$. If we draw the color classes of vertices of $G_{k,\nu}$ as parallel columns with the given order, it follows that reduced expressions correspond to minimal subgraphs of $G_{k,\nu}$ whose edges cover all the vertices in~$S$ without crossing (that is, no pair of edges of the form $(s_i, j)$ and $(s_{i'}, j')$ with $i<i'$ and $j>j'$). Two reduced expressions differ by an element if and only if the corresponding graphs differ by an edge.

 Subgraphs of $G_{k,\nu}$ with the non-crossing property determine the cells in the staircase triangulation of the product of simplices $\Delta_{m-k}\times\Delta_{k}$ (see~\cite[Section~6.2]{DeLoeraRambauSantos2010}). By means of the Cayley trick (cf.~\cite[Section~9.2]{DeLoeraRambauSantos2010}), we obtain a fine mixed subdivision of $(m-k+1)\cdot\Delta_k$ whose cells are in bijection with the non-crossing subgraphs of $G_{k,\nu}$ covering~$S$. In particular, the vertices of this subdivision are in bijection with minimal non-crossing subgraphs of $G_{k,\nu}$ covering~$S$, which correspond to  
 reduced expressions of $\pi$ in $Q$, or equivalently, to the facets of $\subwordComplex{Q}{\pi}$. Two vertices are connected by an edge in the subdivision if and only if the corresponding facets are adjacent in the subword complex.
\end{proof}

Motivated by the proof of Proposition~\ref{prop:multi}, it is natural to consider only lattice paths of the form $\nu=E^k\mu N^k$ because all points strictly south-west (resp. north-east) of the $k$th east step (resp. $k$th north step in reverse order) of $\nu$ are irrelevant.

The following figure illustrates three examples of such paths for $k=2$. For the first path, we get that the $(2,\nu_1)$-Tamari complex can be obtained as the boundary complex of a $4$-dimensional cyclic polytope with $7$ vertices (see Lemma~8.8 in~\cite{PilaudSantos2009}), which is the dual of the simple multiassociahedron $\Delta_{7,2}^*$. It would be interesting to know if the facet adjacency graphs of the $(2,\nu_2)$- and $(2,\nu_3)$-Tamari complexes can be obtained as the edge graphs of some subdivisions of $\Delta_{7,2}^*$.
 
 \begin{center}
 \includegraphics[width=.65\textwidth]{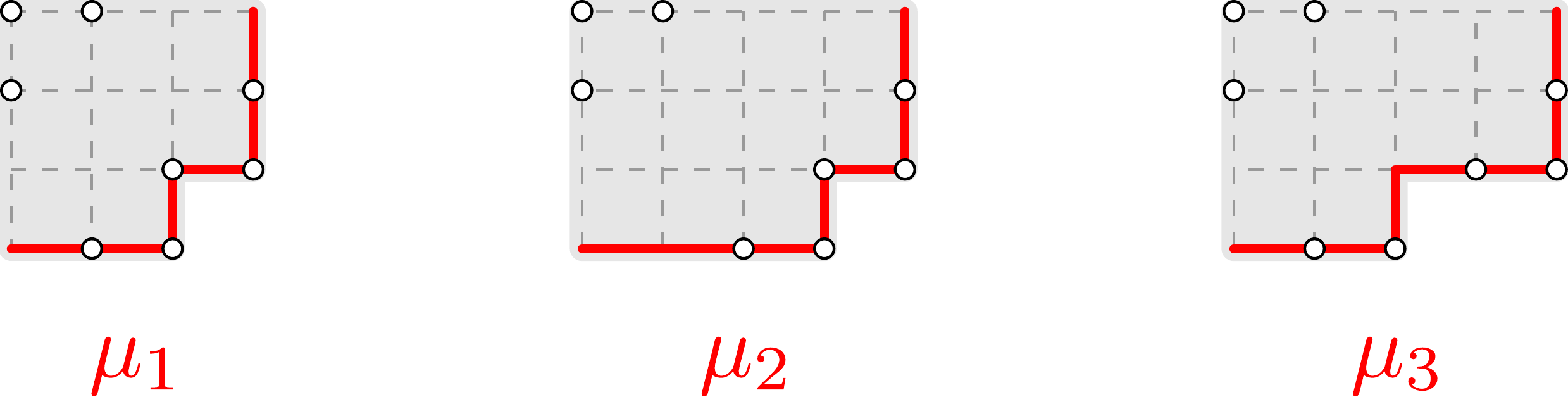}
 \end{center}

\begin{question}
\label{que:multisubd}
 Let $\nu=E^k\mu N^k$ be a lattice path such that $\mu$ does not have two consecutive north steps and does not end with a north step. Is the facet adjacency graph of the $\knu$-Tamari complex realizable as the edge graph of a polytopal subdivision of a simple multiassociahedron $\Delta^*_{m+2k+2,k}$, where $m$ is the number of north steps in~$\mu$?
\end{question}

We have seen in the proof of Propoposition~\ref{prop:multi} that the answer to this question is positive when $\mu$ consists only of east steps. 
 Instances of this result are illustrated in Figures~\ref{fig:mixedStaircase} and~\ref{fig:mixedTetrahedron}.
 \begin{figure}[htpb]
\centering 
\includegraphics[width=.7\textwidth]{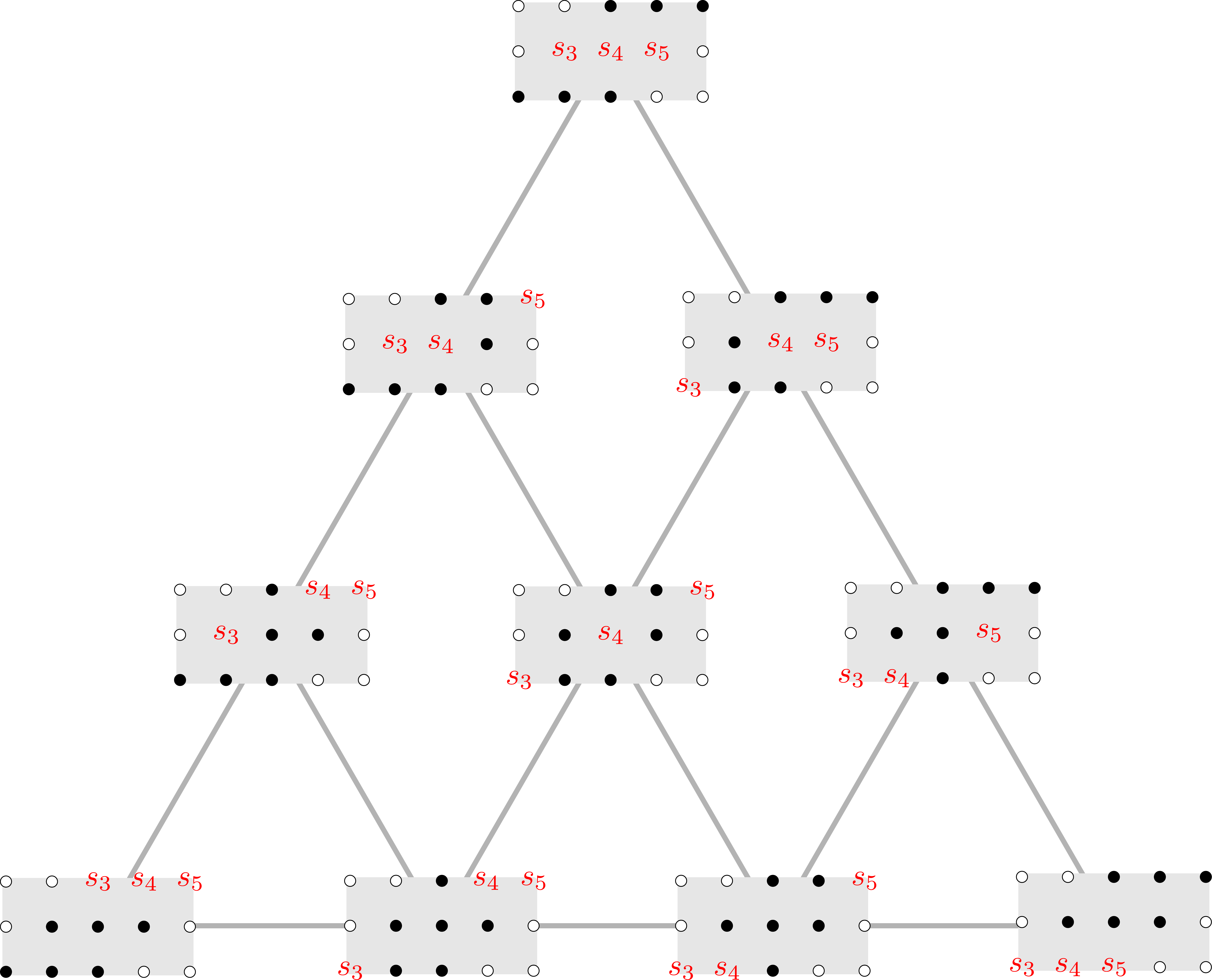}
\caption{For $k=2$ and $\nu=\sfE^4\sfN^2$, the facet adjacency graph of the $(k,\nu)$-Tamari complex is the graph of a mixed subdivision of a triangle. By the proof of Proposition~\ref{prop:multi}, this is equivalent to the facet adjacency graph of the Fuss-Catalan $(k,\nu')$-Tamari complex for $\nu'=(NE^4)^3$.}
 \label{fig:mixedStaircase}
\end{figure}

 \begin{figure}[htpb]
\centering 
\includegraphics[width=.7\textwidth]{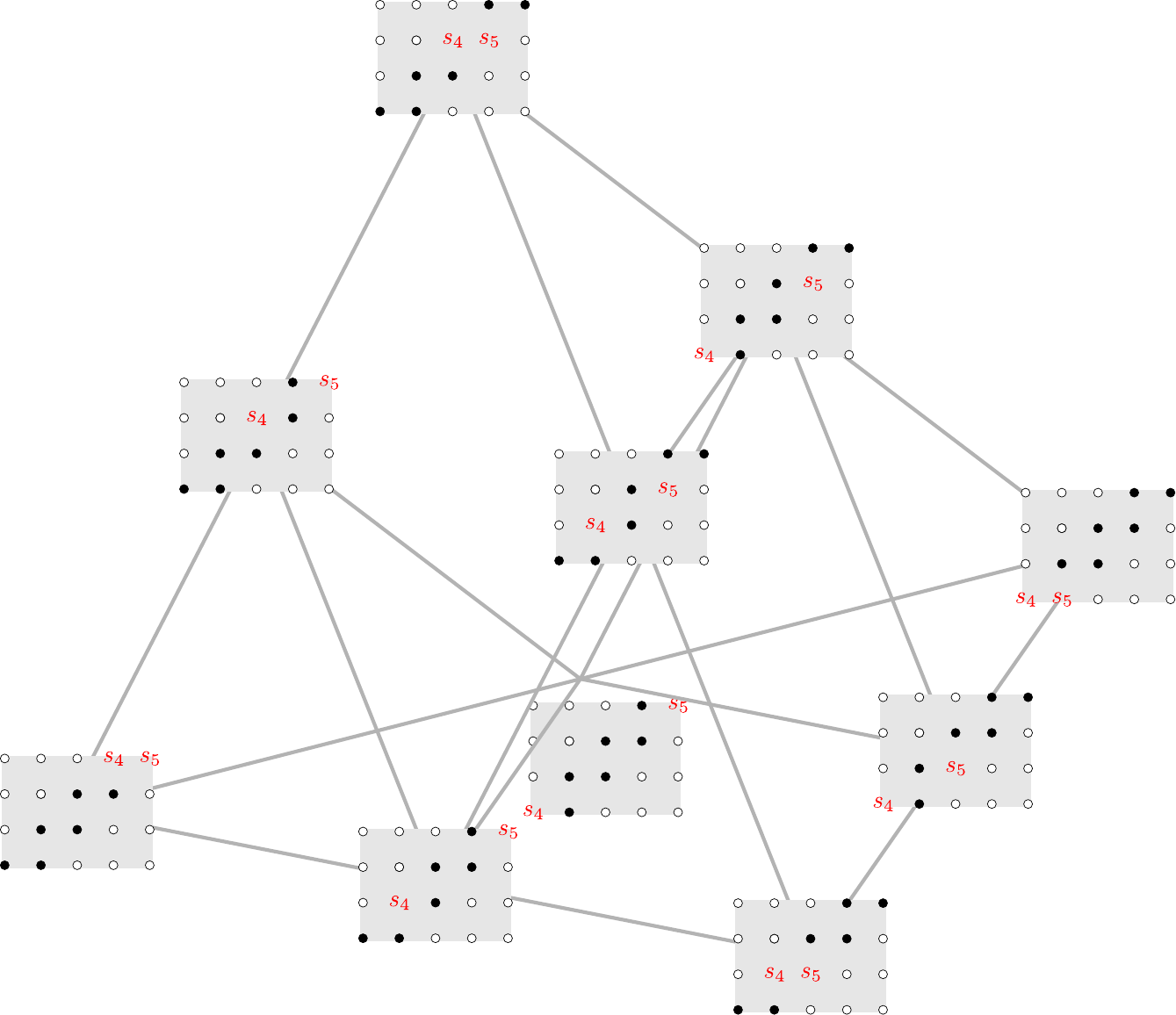}
\caption{For $k=3$ and $\nu=\sfE^4\sfN^3$, the facet adjacency graph of the $(k,\nu)$-Tamari complex is the graph of a mixed subdivision of a tetrahedron. By the proof of Proposition~\ref{prop:multi}, this is equivalent to the facet adjacency graph of the Fuss-Catalan $(k,\nu')$-Tamari complex for $\nu'=(NE^4)^4$.}
 \label{fig:mixedTetrahedron}
\end{figure}

\section*{Acknowledgements}

We thank Xavier Viennot for many valuable comments, and in particular for noticing the correspondence between $\nu$-trees and non-crossing tree-like tableaux.

\bibliographystyle{abbrv}
\bibliography{biblio}

\end{document}